\documentclass[final]{siamart190516}

\usepackage[T1]{fontenc}
\usepackage[utf8]{inputenc}
\usepackage{lmodern}
\usepackage[british]{babel}

\usepackage{graphicx}
\usepackage[export]{adjustbox}

\usepackage{amsmath}
\usepackage{amsfonts}
\usepackage{amssymb}
\usepackage{mathtools}
\usepackage{commath}
\usepackage{nicefrac}
\usepackage{array}
\usepackage{multirow}

\usepackage{bm}
\usepackage{calrsfs}
\usepackage{stmaryrd}
\DeclareMathAlphabet{\pazocal}{OMS}{zplm}{m}{n}
\DeclareMathAlphabet{\pazocalbf}{OMS}{cmsy}{b}{n}
\usepackage{etoolbox}
\preto\subequations{\ifhmode\unskip\fi}

\usepackage{hyperref}

\usepackage{enumitem}
  \setlist[enumerate]{nosep, topsep=0pt, wide = 1em, leftmargin=*}

\newsiamremark{remark}{Remark}
\newsiamthm{assumption}{Assumption}

\newcommand{\clem}[2][]{\ifstrempty{#1}{%
  c_{\hyperref[#2]{L\ref*{#2}}}}{%
  c_{\hyperref[#2]{L\ref*{#2}#1}}}}
\newcommand{\cthm}[2][]{\ifstrempty{#1}{%
  c_{\hyperref[#2]{T\ref*{#2}}}}{%
  c_{\hyperref[#2]{T\ref*{#2}#1}}}}
\newcommand{\cass}[2][]{\ifstrempty{#1}{%
  c_{\hyperref[#2]{\ref*{#2}}}}{%
  c_{\hyperref[#2]{\ref*{#2}#1}}}}
\newcommand{\ceq}[1]{c_{\hyperref[#1]{(\ref*{#1})}}}

\newcommand{\leqc}{\lesssim}
\newcommand{\geqc}{\gtrsim}
\newcommand{\eqc}{\simeq}

\newcommand{\R}{\mathbb{R}}
\newcommand{\N}{\mathbb{N}}
\newcommand{\PP}{\mathbb{P}}

\newcommand{\ub}{\bm{u}}
\newcommand{\ubn}{\ub^n}
\newcommand{\ubnl}{\ubn_l}
\newcommand{\ubh}{\ub_h}
\newcommand{\ubhn}{\ubn_h}
\newcommand{\ubI}{\ub_I}
\newcommand{\ubIn}{\ubI^n}
\newcommand{\vb}{\bm{v}}
\newcommand{\vbh}{\vb_h}

\newcommand{\vbhl}{\vbh^{\ell}}
\newcommand{\wb}{\bm{w}}

\newcommand{\Ub}{\bm{\xi}}
\newcommand{\Ubn}{\Ub^n}
\newcommand{\Ubnl}{\Ubn_l}
\newcommand{\Ubh}{\Ub_h}
\newcommand{\Ubhn}{\Ubh^n}
\newcommand{\Vbb}{\bm{\zeta}}

\newcommand{\C}{\bm{\chi}}
\newcommand{\Ctn}{\C(t^n)}
\newcommand{\Ch}{\C_h}
\newcommand{\Cn}{\C^n}
\newcommand{\Cnl}{\Cn_l}
\newcommand{\Chn}{\Ch^n}

\newcommand{\lamb}{\bm{\lambda}}
\newcommand{\lambn}{\lamb^n}
\newcommand{\lambnl}{\lambn_l}
\newcommand{\lambh}{\lamb_h}
\newcommand{\lambhn}{\lambh^n}
\newcommand{\lambI}{\lamb_I}
\newcommand{\lambIn}{\lambI^n}
\newcommand{\lambt}{\widetilde{\lamb}_h}
\newcommand{\mub}{\bm{\mu}}
\newcommand{\mubh}{\mub_h}
\newcommand{\mubhl}{\mubh^{\ell}}

\newcommand{\Eu}[1][]{\mathbb{E}^{#1}}
\newcommand{\El}[1][]{\mathbb{L}^{#1}}
\newcommand{\EU}[1][]{\mathbb{U}^{#1}}
\newcommand{\EC}[1][]{\mathbb{V}^{#1}}

\newcommand{\II}{\pazocal{I}}
\newcommand{\Iu}{\II^u}
\newcommand{\Il}{\II^\lambda}
\newcommand{\erru}[1][]{\mathbf{e}_h^{#1}}
\newcommand{\errl}[1][]{\mathbf{l}_h^{#1}}
\newcommand{\errIu}[1][]{\bm{\eta}^{#1}}
\newcommand{\errIl}[1][]{\bm{\theta}^{#1}}

\newcommand{\Ern}{\mathfrak{E}^n}
\newcommand{\Ecn}{\Ern_\text{C}}
\newcommand{\EIn}{\Ern_\text{I}}

\newcommand{\Vb}{\bm{V}}
\newcommand{\Vbn}{\Vb^n}
\newcommand{\Vbh}{\Vb_h}
\newcommand{\Vbhn}{\Vbh^n}
\newcommand{\Nb}{\bm{N}}
\newcommand{\Nbn}{\Nb^n}
\newcommand{\Nbhn}{\Nbn_h}

\newcommand{\Hb}[3]{\pazocalbf{H}^{#1}_{#2}(#3)}
\newcommand{\Lb}[3]{\pazocalbf{L}^{#1}_{#2}(#3)}

\newcommand*{\Wb}[2]{\pazocalbf{W}^{#1}(#2)}

\renewcommand{\O}{\Omega}
\newcommand{\Oh}{\O_{h}}
\newcommand{\Ohn}{\Oh^{n}}
\newcommand{\On}{\O^{n}}
\newcommand{\Ot}{\widetilde\O}
\newcommand{\G}{\Gamma}
\newcommand{\Gh}{\G_{h}}
\newcommand{\Ghn}{\Gh^{n}}
\newcommand{\Gn}{\G^{n}}
\newcommand{\Gout}{\G_\text{out}}
\newcommand{\QQ}{\pazocal{Q}}

\newcommand{\OO}{\pazocal{O}}
\newcommand{\OT}{\OO_{\mathcal{T}}}

\newcommand{\Od}[1]{\OO_{\delta}(#1)}
\newcommand{\Odh}[1]{%
  \mathchoice{\OO_{\delta_{h}}(#1)}
  {\OO_{\delta_{h}}(#1)}
  {\OO_{\delta_{\adjustbox{scale=0.6}{$\scriptstyle h$}}}(#1)}
  {}}
\newcommand{\OdT}[2][\mathcal{T}]{%
  \mathchoice{\OO^{#2}_{\delta_{h}, #1}}
  {\OO^{#2}_{\delta_{h}, #1}}
  {\OO^{#2}_{\delta_{\adjustbox{scale=0.6}{$\scriptstyle h$}}, #1}}
  {}}
\newcommand{\OTn}{\OT^{n}}
\newcommand{\OGTn}{\OO_{\Ghn}^n}

\newcommand{\Scal}{\pazocal{S}}
\newcommand{\Spm}{\Scal^\pm}
\newcommand{\Sp}{\Scal^{+}}

\newcommand{\hmax}{h_{\text{max}}}

\newcommand{\Tht}{\widetilde{\mathcal{T}}_h}

\newcommand{\Thn}{\mathcal{T}^{n}_{h}}
\newcommand{\Thnd}{\mathcal{T}^{n}_{h,\delta_h}}
\newcommand{\ThGhn}{\mathcal{T}^{n}_{h,\Ghn}}

\newcommand{\ThnSpm}{\mathcal{T}^n_{h,\Spm}}
\newcommand{\ThnSp}{\mathcal{T}^n_{h,\Sp}}

\newcommand{\scalesixtypercent}[1]{\adjustbox{scale=0.6}{$\scriptstyle #1$}}
\newcommand{\Fhnd}{%
  \mathchoice{\mathcal{F}^n_{h,\delta_h}}
   {\mathcal{F}^n_{h,\delta_h}}
   {\mathcal{F}^n_{h,\delta_{\scalesixtypercent{h}}}}
   {}}

\newcommand{\wF}{\omega{_F}}

\newcommand{\gsu}{\gamma_{gp}}
\newcommand{\gsl}{\gamma_{\lambda}}

\newcommand{\gb}{\bm{g}}
\newcommand{\Fb}{\bm{F}}
\newcommand{\Fbh}{\Fb_h}
\newcommand{\Fbhn}{\Fbh^n}

\newcommand{\xb}{\bm{x}}
\newcommand{\nb}{\bm{n}}

\newcommand{\wninf}{\bm{w}^{\nb}_{\infty}}
\newcommand{\tend}{t_\text{end}}
\newcommand{\dt}{\Delta t}

\newcommand{\cphininv}{\circ(\Phi^n)^{-1}}
\newcommand{\cphin}{\circ\Phi^n}

\newcommand{\jump}[1]{\llbracket#1\rrbracket}
\newcommand{\restr}[2]{{\left.\kern-\nulldelimiterspace#1\right|_{#2}}}
\newcommand{\Ex}{\pazocal{E}}

\newcommand{\lprod}[3]{\big(#1,#2\big)_{#3}}
\newcommand{\nrm}[2]{\big\Vert#1\big\Vert_{#2}}
\newcommand{\tripnrm}[1]{\mathopen{|\mkern-1.5mu|\mkern-1.5mu|}%
  #1\mathclose{|\mkern-1.5mu|\mkern-1.5mu|}}
\newcommand{\tnrms}[2][n]{\tripnrm{#2}_{\ast,#1}}
\newcommand{\myplus}{\mkern-1.5mu+\mkern-1.5mu}

\newcommand{\MF}{\mathfrak{M}}
\newcommand{\TT}{\mathfrak{T}}
\newcommand{\eps}{\varepsilon}

\DeclareMathOperator{\meas}{meas}
\DeclareMathOperator{\dist}{dist}
\DeclareMathOperator{\esssup}{ess\,sup}
\DeclareMathOperator{\id}{id}

\title{Error analysis for a parabolic PDE model problem on a coupled moving domain in a fully Eulerian framework
\thanks{Submitted to the editors DATE.
\funding{This work was funded by the German Science Foundation (DFG) within the project 314838170, GRK 2297 MathCoRe, HvW acknowledges support through Austrian Science Fund (FWF) project F65.}}
}
\author{Henry von Wahl\thanks{Institute for Mathematics, Universität Wien, Austria (\email{henry.wahl@univie.ac.at})}
\and
Thomas Richter\thanks{Institute for Analysis and Numerics, Otto-von-Guericke-Universität Magdeburg, Germany (\email{thomas.richter@ovgu.de})}
}
\headers{Eulerian error estimate on a coupled moving domain}{H. von Wahl and T. Richter}

\ifpdf
\hypersetup{
  pdftitle={Error analysis for a parabolic PDE model problem on a coupled moving domain in a fully Eulerian framework},
  pdfauthor={H. von Wahl and T. Richter}
}
\fi

\begin{document}
\maketitle

\begin{abstract}
We introduce an unfitted finite element method with Lagrange-multipliers to
study an Eulerian time stepping scheme for moving domain problems applied to
a model problem where the domain motion is implicit to the problem. We consider
a parabolic partial differential equation (PDE) in the bulk domain,
and the domain motion is described by an ordinary differential equation (ODE),
coupled to the bulk partial differential
equation through the transfer of forces at the moving interface. The
discretisation is based on an unfitted finite element discretisation on a
time-independent mesh. The method-of-lines time discretisation is enabled by
an implicit extension of the bulk solution through additional stabilisation,
as introduced by Lehrenfeld \& Olshanskii (ESAIM: M2AN, 53:585--614, 2019).
The analysis of the coupled problem relies on the Lagrange-multiplier
formulation, the fact that the Lagrange-multiplier solution is equal to the
normal stress at the interface and that the motion of the interface is given
through rigid body motion. This paper covers the complete stability analysis
of the method and an error estimate in the energy norm, under an assumption
on the discrete interface velocity. This includes the
dynamic error in the domain motion resulting from the discretised ODE and the
forces from the discretised PDE. To the best of our knowledge this is the first
error analysis of this type of coupled moving domain problem in a fully
Eulerian framework. Numerical examples illustrate the theoretical results.
\end{abstract}



\section{Introduction}
\label{sec.intro}

Particulate flows, particle settling and in the broader sense fluid
solid interactions play a major role in applications,
ranging from medicine~\cite{Dong2017,Ric17,FailerMinakowskiRichter2021} and
biology~\cite{Maggi2013} to industry~\cite{BazilevsTakizawaTezduyar2013,
Sundaresan2003}.

The most well-established method to solve the resulting fluid-structure
interaction problem is the so-called Arbitrary Lagrangian-Eulerian (ALE)
method~\cite{DGH82}. Here a mesh of a reference geometry is created,
and the moving domain problem is solved by mapping the equations into the
reference configuration. A significant burden in this approach occurs when the
deformation with respect to the reference configuration becomes very large. In
this case, re-meshing procedures~\cite{ShamanskiySimeon21} must be
included, or Eulerian approaches~\cite{Peskin1972,Richter2013} need to be
considered. In this paper, we shall focus on the latter approach. In
particular, we shall focus on an unfitted Eulerian approach in the context of
fluid-rigid body interactions. Such Eulerian approaches are based on a fixed
background mesh to define a set of potential unknowns, and the geometry of the
problem is described separately.

The main challenge in Eulerian approaches for time-dependent moving domain
problems is the approximation of the time-derivative. Standard
approximations based on finite differences are not easily applicable since
the expression $\partial_t u\approx (u^n - u^{n-1})/\dt$ is not well-defined
if $u^{n-1}$ and $u^n$ live on different domains. A successful approach to
deal with this challenge is a class of space-time Galerkin formulation in an
Eulerian setting. This approach has been proven to work for scalar bulk
problems~\cite{LR13,Leh15,Pre18,Zah18}, problems on moving
surfaces~\cite{OR14,ORX14} and coupled bulk-surface problems~\cite{HLZ16}.
However, space-time Galerkin methods have the draw-back, that a
higher-dimensional problem has to be solved. This problem can be circumvented
by an approach using adjusted quadrature rules to reduce the space-time
problem into a classical time stepping scheme~\cite{FR17}. However, this comes
at the expense of costly computations of projections between different function
spaces.

In this paper, we shall follow a different approach that recovers the use of
standard time stepping schemes by using an extension of the previous
solution to the domain of the next time step.
This concept was first introduced in~\cite{OX17} for problems on moving
surfaces and then for scalar bulk convection-diffusion problems in~\cite{LO19}. 
The essential idea in the latter is to apply additional stabilisation in a 
strip around the moving interface, such that the discrete solution $u^{n-1}_h$
is well-defined in a larger, non-physical domain $\Odh{\Oh^{n-1}}\supset\Ohn$.
As a result, the expression $(u_h^n - u_h^{n-1})/\dt$ is again well-defined on 
the domain $\Ohn$.

This unfitted finite element method with Eulerian time stepping schemes for
partial differential equation problems posed on moving domains has so far been
considered for problems where the motion of the domain is a given
quantity~\cite{LO19,BFM19,vWRL20,LL21,AB21}. Furthermore, the method developed
in these papers has been successfully applied to a fluid-structure
interaction problem, where the geometry motion is part of the problem to be
solved~\cite{vWRFH21,vWR21}. However, no error analysis is available for this
setting. 

The main contribution of this paper is the development of an error estimate
for this Eulerian time stepping scheme for a partial differential equation 
(PDE) in a moving domain, where the domain motion is driven by an ordinary
differential equation (ODE) coupled to the PDE. To this end, we consider a set
of simplified equations to analyse this kind of Eulerian time stepping with
coupled domain motion. This will be a parabolic PDE in the time-dependent
bulk domain, while the motion of the moving interface is driven by
translational rigid body motion. These two equations are then coupled on the
moving interface by the non-homogeneous Dirichlet boundary conditions and the
forces acting on the moving interface. The coupling condition is the
same that would be typical for standard fluid-structure interactions
problems~\cite{Ric17}. The main simplification is the restriction to a 
parabolic PDE model, which allows us to avoid the additional 
difficulties that would be involved in treating the divergence
constraint. This restriction allows us a clearer presentation of the
nevertheless technically complex proofs. We assume that an extension
to the Stokes equations would not bring any significant surprises.    

The remainder of this paper is structured as follows. In \cref{sec.model},
we discuss the mathematical model under consideration and show the unique
solvability thereof. We then begin by a temporal semi discretisation of the
problem in \cref{sec.temp-discr} and show the stability of the resulting
scheme under an assumption on the discrete interface velocity. \Cref{sec.discrete} then covers the
full discretisation of our problem. We introduce our CutFEM
Lagrange-multiplier discretisation, then show the discrete problem's
solvability and stability of the discrete scheme. We then quantify the error
in the time-dependent geometry resulting from the discretisation of the ODE
governing the motion of the domain. This is then used to prove a consistency
error estimate and finally an error estimate in the energy norm. In
\cref{sec.num-examples}, we illustrate our theoretical results with some
numerical examples, including extensions to higher-order in both space and
time. Finally, we give a brief summary of the results and an
outlook for potential future work in \cref{sec.conclusions}.

\section{Mathematical Problem}
\label{sec.model}

Let $\Ot\subset\R^d$, for $d\in\{2,3\}$, be an open bounded domain, which we
denote as the background domain. We divide $\Ot$ into the $d$-dimensional open
bulk domain of interest $\O(t)$, the $d$-dimensional complement of $\Ot$ in
$\Ot$ denoted as $\Sigma(t)= \Ot\setminus\overline\O(t)$ and $d-1$-dimensional
interface $\Gamma(t)=\partial\Sigma$ between the two, i.e., $\Ot = \O(t) 
\dot\cup \G(t) \dot\cup\Sigma(t)$. We assume that the interface $\G(t)$ can be
described by a smooth level set function 
$\phi(t,\xb)$, i.e,
\begin{equation*}
  \G(t) = \{\xb\in\Ot \;\vert\; \phi(t, \xb) = 0 \}
  \qquad\text{and}\qquad
  \O(t) = \{\xb\in\Ot \;\vert\; \phi(t, \xb) < 0 \}.
\end{equation*}
Furthermore, we denote the fixed part of the boundary of the bulk
domain as $\Gout=\partial\Ot$. A sketch of such a domain can be seen in
\cref{fig.intro:domain-sketch}. Now let $[0,\tend]$ be a finite time
interval and assume that $\O(t)\subset\Ot$ for all $t\in[0,\tend]$.
We then define the space-time domain
\begin{equation*}
  \QQ\coloneqq\bigcup_{t\in(0,\tend)}\O(t)\times\{t \}.
\end{equation*}
Let $\ub$ denote the solution in the bulk domain $\O$, $\Ub$ the velocity of
the interface $\G$ and $\C$ the centre of mass of $\Sigma$ relative to the
initial position.
In $\QQ$, we then consider the vector-valued parabolic model problem
\begin{subequations}\label{eqn:system.strong-pde}
  \begin{align}
    && \partial_t\ub - \Delta\ub &= \bm{0}& \text{in }& \O(t)\times(0, \tend]\\
    && \ub &= \Ub                         & \text{in }& \G(t)\times(0, \tend]\\
    && \ub &= \bm{0}                      & \text{in }& \Gout\times(0, \tend]\\
    && \ub(0) &= \ub_0                    & \text{in }&\O(0)
  \end{align}
\end{subequations}
where the motion of the moving interface $\G(t)$ is determined by
\begin{subequations}\label{eqn:system.strong-ode}
  \begin{align}
    && \frac{\dif}{\dif t}\Ub &= \gb + \Fb& \text{in } &(0, \tend]\\
    && \frac{\dif}{\dif t}\C &= \Ub & \text{in } &(0, \tend]\\
    && \Ub(0) &= \Ub_0\\
    && \C(0) &= \bm{0}
  \end{align}
\end{subequations}
Here, $\ub_0$ and $\Ub_0$ are given initial conditions, $\gb$ is a constant external
force acting on $\Sigma$, and $\Fb = - \int_{\G(t)} \partial_{\nb}\ub\dif s$ is the
force acting from $\O$ onto $\Sigma$, and $\G$ moves with velocity
$\Ub$ through $\Ot$. The system \eqref{eqn:system.strong-pde} can be seen as a
simplification of the transient Stokes equations on a moving
domain~\cite{vWRL20,BFM19} by restricting the velocity to the space of
divergence-free functions, but with the added complexity of the motion being
driven by an ODE, coupled to the bulk equations through the transfer forces.
We consider this model problem, since this already shows the difficulties in
the error analysis of a coupled moving domain problem in a purely Eulerian
framework.

The more complicated problem of coupling a rigid body to the non-linear Navier-Stokes equations has been studied extensively. In~\cite{DesjardinsEsteban1999} the authors show the existence of weak solutions which are global in time up to collision of the rigid body with the boundary. In three spatial dimensions smallness of the data is required. Strong solutions are studied in~\cite[Theorem 2.2]{Takahashi2003} and their unique existence, global in time up to collision, is shown in the two dimensional case. In the three dimensional case, solutions that are global in time are shown to exist under smallness requirements on the data. For the linearised case, the authors of~\cite{Maity2017} could even show maximal regularity. 

\begin{figure}
  \centering
  \includegraphics{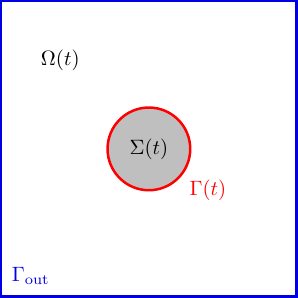}
  \caption{Sketch of an example background domain, divided into the bulk domain of interest $\O(t)$, the complementary domain $\Sigma(t)$, the moving interface $\G(t)$ and fixed outer boundary $\Gout$.}
  \label{fig.intro:domain-sketch}
\end{figure}

\subsection{Stability estimate}
\label{sec.model:subsec.solvability}

We begin by showing that the system \eqref{eqn:system.strong-pde}--
\eqref{eqn:system.strong-ode} satisfies a stability estimate, depending only
on the problem data. Let us consider the spaces
\begin{equation*}
  \Vb(t) = \{\vb\in\Hb{1}{}{\O(t)}\;\vert\; \restr{\vb}{\Gout} = 0 \}
  \qquad\text{and}\qquad
  \Nb(t) = \Hb{-1/2}{}{\G(t)}.
\end{equation*}
Multiplying \eqref{eqn:system.strong-pde} with test-functions from the
appropriate spaces and by using integration by parts we get the weak
formulation: Find $(\ub(t),\lamb(t),\Ub(t),\C(t))\in
\Vb(t)\times\Nb(t)\times\R^d\times\R^d$ such that
\begin{multline}\label{eqn.model:weak-form}
  \lprod{\partial_t\ub}{\vb}{\O(t)} + \lprod{\nabla\ub}{\nabla\vb}{\O(t)}
  + \lprod{\lamb}{\vb}{\G(t)} + \lprod{\mub}{\ub -\Ub}{\G(t)}\\
  + \lprod{\frac{\dif}{\dif t}\Ub}{\Vbb_1}{2} 
  + \lprod{\frac{\dif}{\dif t}\C}{\Vbb_2}{2}
  = \lprod{\gb + \Fb}{\Vbb_1}{2} + \lprod{\Ub}{\Vbb_2}{2}
\end{multline}
holds for all $(\vb,\mub,\Vbb_1, \Vbb_2)\in
\Vb(t)\times\Nb(t)\times\R^d\times\R^d$. Furthermore, the solution
of the Lagrange-multiplier is $\lamb =- \partial_{\nb}\ub$,
see \cite[Theorem~3.2]{Bab73}.

\begin{lemma}\label{lemma.cont-stability}
For the velocity and position solution $(\ub,\Ub, \C)$ of 
\eqref{eqn.model:weak-form}, it holds that
\begin{multline*}
  \nrm{\ub(t)}{\O(t)}^2 + \nrm{\Ub(t)}{2}^2 + \nrm{\C(t)}{2}^2
  + \int_{0}^t\nrm{\nabla\ub(s)}{\O(s)}^2\dif s\\
  \leq \exp(t)\left[\nrm{\ub_0}{\O(0)}^2 + \nrm{\Ub_0}{2}^2\right] +
   \clem{lemma.cont-stability}(\exp(t) - 1)\nrm{\gb}{2}^2,
\end{multline*}
with a constant $\clem{lemma.cont-stability}>0$ that only depends on the domain size
$\vert\O\vert$ and $\vert\G\vert$.
\end{lemma}
\begin{proof}
Testing \eqref{eqn.model:weak-form} with $(\vb,\mub,\Vbb_1, \Vbb_2)=
(\ub,-\lamb,\Ub, \C)$ gives
\begin{multline}\label{eqn.model:weakform.tested}
  \underbrace{\lprod{\partial_t\ub}{\ub}{\O(t)}}_{\TT_1}
  + \nrm{\nabla\ub}{\O(t)}^2
  \underbrace{-\lprod{\partial_{\nb}\ub}{\Ub}{\G(t)}}_{\TT_2}
  + \underbrace{\lprod{\frac{\dif}{\dif t} \Ub}{\Ub}{2}}_{\TT_3}
  + \underbrace{\lprod{\frac{\dif}{\dif t} \C}{\C}{2}}_{\TT_4}\\
    = \underbrace{\lprod{\gb+\Fb}{\Ub}{2}}_{\TT_5}
     + \underbrace{\lprod{\Ub}{\C}{2}}_{\TT_6}.
\end{multline}
Using the Reynolds transport theorem for moving domains and the fact
that $\Ub$ is both the velocity of the moving interface and the trace of
$\ub$ on this interface, we have
\begin{equation*}
  \TT_1 = \lprod{\partial_t\ub}{\ub}{\O(t)}
    = \frac{1}{2}\Big(\frac{\dif}{\dif t}\nrm{\ub}{\O(t)}^2
      - \int_{\G(t)}(\ub\cdot\ub)\Ub\cdot\nb\dif s \Big).
\end{equation*}
Now, since we have that $\restr{\ub}{\G}=\Ub$, we may interchange $\Ub$ and
$\ub$ in the integral over the interface. Using the fact that $\Ub$ is constant
in space, we then find using the divergence theorem that
\begin{equation*}
  -\int_{\G(t)}(\ub\cdot\ub)\Ub\cdot\nb\dif s
   = -\nrm{\Ub}{2}^2\int_{\G(t)}\Ub\cdot\nb\dif s
   = -\nrm{\Ub}{2}^2\int_{\O(t)}\nabla\cdot\Ub\dif\xb
   = 0.
\end{equation*}
As a result, we have
\begin{equation}\label{eqn.model:cont-est:T1}
  \lprod{\partial_t\ub}{\ub}{\O(t)}
    = \frac{1}{2}\frac{\dif}{\dif t}\nrm{\ub}{\O(t)}^2.
\end{equation}
Again, using the fact that $\Ub$ is constant in space, we find that
\begin{equation}\label{eqn.model:cont-est:T2}
  \TT_2 = -\lprod{\partial_{\nb}\ub}{\Ub}{\G(t)}
    = - \int_{\G(t)}\partial_{\nb}\ub\dif s \cdot \Ub = \lprod{\Fb}{\Ub}{2}.
\end{equation}
This then cancels with the drag contribution on the right-hand side of
\eqref{eqn.model:weakform.tested} in $\TT_5$. For the third and fourth term, we
immediately have
\begin{equation}\label{eqn.model:cont-est:T3-4}
  \TT_3 = \lprod{\frac{\dif}{\dif t} \Ub}{\Ub}{2}
    = \frac{1}{2}\frac{\dif}{\dif t}\nrm{\Ub}{2}^2
  \qquad\text{and}\qquad
  \TT_4 = \lprod{\frac{\dif}{\dif t} \C}{\C}{2}
    = \frac{1}{2}\frac{\dif}{\dif t}\nrm{\C}{2}^2.
\end{equation}
Under our assumption that $\gb$ is constant in space and using that
$\ub=\Ub$ on $\G(t)$, we can rewrite the first part of the fourth term as an 
integral over $\G(t)$, i.e.,
$\lprod{\gb}{\Ub}{2} = \frac{1}{\vert\G\vert}\lprod{\gb}{\ub}{\G(t)}$.
Using the trace and Poincaré estimates, we then find
\begin{equation*}
  \vert \lprod{\gb}{\Ub}{2} \vert
    \leq \vert\G\vert^{-\frac{1}{2}} \nrm{\gb}{2}
      c_{\O}\nrm{\ub}{\Hb{1}{}{\O(t)}}
    \leq \vert\G\vert^{-\frac{1}{2}} \nrm{\gb}{2}
      c_{\O}c_P \nrm{\nabla\ub}{\O(t)},
\end{equation*}
where with an abuse of notation, we set $c_P = \max\{2, c_P\}$.
Note that the Poincaré inequality is applicable, since $\restr{\ub}{\Gout}=0$,
c.f.~\cite[Remark~A.37]{Joh16}. With a weighted Young's inequality, we then
have
\begin{equation}\label{eqn.model:cont-est:T5a}
   \vert \lprod{\gb}{\Ub}{2} \vert
    \leq \frac{c_\O^2c_P^2}{2\vert\G\vert}\nrm{\gb}{2}^2
    + \frac{1}{2}\nrm{\nabla\ub}{\O(t)}^2.
\end{equation}
For the final term, we use the Cauchy-Schwarz and Young's inequalities to 
estimate
\begin{equation}\label{eqn.model:cont-est:T6}
  \TT_6 = \lprod{\Ub}{\C}{2} \leq \nrm{\Ub}{2}\nrm{\C}{2}
    \leq \frac{1}{2}\nrm{\Ub}{2}^2 + \frac{1}{2}\nrm{\C}{2}^2.
\end{equation}
We insert
\eqref{eqn.model:cont-est:T1}, \eqref{eqn.model:cont-est:T2},
\eqref{eqn.model:cont-est:T3-4}, \eqref{eqn.model:cont-est:T5a}
\eqref{eqn.model:cont-est:T6} into \eqref{eqn.model:weakform.tested} to get
\begin{equation*}%
  \frac{\dif}{\dif t}\nrm{\ub}{\O(t)}^2
    + \nrm{\nabla\ub}{\O(t)}^2 + \frac{\dif}{\dif t}\nrm{\Ub}{2}^2
    + \frac{\dif}{\dif t}\nrm{\C}{2}^2
  \leq \frac{c_\O^2 c_P^2}{\vert\G\vert}\nrm{\gb}{2}^2
    + \nrm{\Ub}{2}^2 + \nrm{\C}{2}^2.
\end{equation*}
Using a version of Gronwall's lemma in differential form,
see~\cite[Lemma~A.55]{Joh16}, and setting $\clem{lemma.cont-stability} 
\coloneqq c_\O^2 c_P^2/\vert\G\vert$ proves the claim. We note
that the dependence on the domain in $\clem{lemma.cont-stability}$ is through
the $d$-dimensional measure of $\O$ and $d-1$-dimensional measure of $\G$,
see~\cite{Gal11} for details. Since our problem only contains rigid body
motion, this is constant and does not depend on the solution.
\end{proof}

\section{Discretisation in Time}
\label{sec.temp-discr}

As a first step, we consider the temporal semi-discretisation of
\eqref{eqn:system.strong-pde}--\eqref{eqn:system.strong-ode} in an Eulerian
framework. To this end, let us consider a uniform time step $\dt \coloneqq
\tend / N$ for some $N\in\N$ and denote $t^n = n\dt$. We define the 
$\delta$-neighbourhood of $\O(t)$ as
\begin{equation*}
  \Od{\O(t)} \coloneqq \{ \xb \in \Ot \;\vert\; \dist(\xb,\O(t))\leq \delta\}.
\end{equation*}
As in~\cite{LO19,BFM19,vWRL20}, the Eulerian time stepping method requires
$\delta$ to be sufficiently large, such that the domain $\O(t^n)$ is a subset
of the $\delta$-neighbourhood of the the previous time step, i.e,
\begin{equation}\label{eqn.temp-discr:domain-inclusion}
  \O(t^n) \subset \Od{\O(t^{n-1})},\quad\text{for }n=1,\dots,N.
\end{equation}
In the aforementioned literature, the motion of the interface was a known
quantity, such that the relation \eqref{eqn.temp-discr:domain-inclusion} is
guaranteed by setting $\delta$ proportional to the maximal interface
normal speed and the time step. In our case, the motion of the interface is
an additional unknown in the system. However, since we know by
\cref{lemma.cont-stability} that the interface-velocity solution is bounded,
we make the following assumption:

\subsection{Temporal Discretisation}
\label{sec.temp-discr:subsec.method}

To enable our Eulerian time stepping, we need a suitable extension operator.
For this extension operator, we require the following family of space-time
anisotropic spaces
\begin{equation*}
  \def\arraystretch{1.2}
  \Lb{\infty}{}{0,T;\Hb{m}{}{\O(t)}}\coloneqq
  \left\{\vb\in\Lb{2}{}{\QQ}\middle\vert\!\!
    \begin{array}{l}
      \vb(\cdot,t)\in\Hb{m}{}{\O(t)}\text{ for a.e. }t\in(0,T)\\
      \quad\text{ and } \esssup_{t\in(0,T)}
        \nrm{\vb(\cdot,t)}{\Hb{m}{}{\O(t)}}<\infty
    \end{array}\!\!
  \right\},
\end{equation*}
for $m=0,\dots,k+1$. We then denote $\partial_t\vb=\vb_t$ as the weak partial
derivative with respect to the time variable, if this exists as an element of
the space-time space $\Lb{2}{}{\QQ}$.
We now assume the existence of a spatial extension operator
\begin{equation*}
    \Ex:\Lb{2}{}{\O(t)} \rightarrow \Lb{2}{}{\Od{\O(t)}},
\end{equation*}
which fulfils the following properties:

\begin{assumption}\label{assump.eulerian:extension}
Let $\vb\in\Lb{\infty}{}{0,T;\Hb{k+1}{}{\O(t)}}\cap\Wb{2,\infty}{\QQ}$ and
$\delta>0$. There exist positive constants $\cass[a]{assump.eulerian:extension}$,
$\cass[b]{assump.eulerian:extension}$ and $\cass[c]{assump.eulerian:extension}$
that are uniform in $t$ such that
\begin{subequations}
    \begin{align}
        \nrm{\Ex\vb}{\Hb{k}{}{\Od{\O(t)}}} &\leq
            \cass[a]{assump.eulerian:extension}\nrm{\vb}{\Hb{k}{}{\O(t)}}
            \label{eqn.eulerian:ExtentionExtimate}\\
        \nrm{\nabla(\Ex\vb)}{\Od{\O(t)}} &\leq
            \cass[b]{assump.eulerian:extension}\nrm{\nabla\vb}{\O(t)}
            \label{eqn.eulerian:GradExtentionExtimate}\\
        \nrm{\Ex\vb}{\Wb{2,\infty}{\Od{\QQ}}} &\leq
            \cass[c]{assump.eulerian:extension}\nrm{\vb}{\Wb{2,\infty}{\QQ}}
            \label{eqn.eulerian:W2ExtentionExtimate}
    \end{align}
\end{subequations}
holds. Furthermore, if for $\vb\in\Lb{\infty}{}{0,T;\Hb{k+1}{}{\O(t)}}$ it
holds for the weak partial time-derivative that
$\vb_t\in\Lb{\infty}{}{0,T;\Hb{k}{}{\O(t)}}$, then
\begin{equation}\label{eqn.eulerian:time-deriv-extension}
    \nrm{(\Ex\vb)_t}{\Hb{k}{}{\Od{\O(t)}}} \leq
    \cass[d]{assump.eulerian:extension}\left[\nrm{\vb}{\Hb{k+1}{}{\O(t)}} +
    \nrm{\vb_t}{\Hb{k}{}{\O(t)}}\right],
\end{equation}
where the constant $\cass[d]{assump.eulerian:extension}>0$ again only
depends on the motion of the spatial domain.
\end{assumption}

Such an extension operator can be constructed explicitly from the classical
linear and continuous \emph{universal} extension operator for Sobolev spaces
(see, e.g.,~\cite[Section VI.3]{Ste70}), when the motion of the domain is
described by a diffeomorphism $\bm{\Psi}(t)\colon\O_0\rightarrow\O(t)$ for each
$t\in[0,T]$ from the reference domain $\O_0$ that is smooth in time.
See~\cite{LO19} for details thereof. Although the motion of the domain is not
given a priori here, we assume that the resulting motion is sufficiently
smooth.

For the weak formulation of the semi-discrete problem, let us consider 
the spaces $\Vbn\coloneqq\Vb(t^n)$ and $\Nbn\coloneqq\Nb(t^n)$.
The temporal semi-discrete weak formulation of our scheme then reads as
follows: Given compatible initial data $(\O_0, \ub_0, \Ub_0)$, i.e., 
$\restr{\ub_0}{\G(0)} = \Ub_0$, for $n=1,\dots,N$ find
$(\ubn,\lambn,\Ubn,\Cn) \in \Vbn\times\Nbn\times\R^d\times\R^d$ such that
\begin{multline}\label{eqn.temp-semi:weak-form}
  \lprod{\frac{1}{\dt}(\ubn - \Ex\ub^{n-1})}{\vb}{\On}
  + \lprod{\nabla\ubn}{\nabla\vb}{\On}
  + \lprod{\lambn}{\vb}{\Gn} + \lprod{\mub}{\ubn - \Ubn}{\Gn}\\
  + \lprod{\frac{1}{\dt}(\Ubn - \Ub^{n-1})}{\Vbb_1}{2}
  + \lprod{\frac{1}{\dt}(\Cn - \C^{n-1})}{\Vbb_2}{2}
  = \lprod{\Fb^n + \gb}{\Vbb_1}{2} + \lprod{\Ub^n}{\Vbb_2}{2}
\end{multline}
holds for all $(\vb, \mub, \Vbb_1, \Vbb_2) \in
\Vbn\times\Nbn\times\R^d\times\R^d$.

In order to specify the appropriate choice of $\delta$ in the above method,
we require the following assumption.

\begin{assumption}\label{assumption.discrete-velocity}
We assume that the time step $\dt>0$ is sufficiently small, such that there
exists a constant $\cass{assumption.discrete-velocity}>1$ with
\begin{equation*}
  \wb^{\nb,\dt}_{\infty}
  \leq \cass{assumption.discrete-velocity} \wninf ,
\end{equation*}
where $\wb^{\nb,\dt}_{\infty}=\max_{i=1,\dots N}
\nrm{\Ub^i \cdot \nb}{\pazocal{L}^\infty(\G(t))}$ is the maximal
interface velocity resulting from the temporally semi-discretised scheme,
$\wninf  = \max_{t\in[0,\tend]}
\nrm{\Ub(t) \cdot \nb}{\pazocal{L}^\infty(\G(t))}$ is the maximal interface
velocity of the smooth problem \eqref{eqn.model:weak-form} and $\nb$ is the
outward pointing unit normal vector vector on $\G(t)$.
\end{assumption}

With \cref{assumption.discrete-velocity}, we then set
\begin{equation}\label{eqn.assumption.discrete-velocity}
  \delta = c_{\delta} \dt \wninf,
\end{equation}
with $c_\delta > \cass{assumption.discrete-velocity} > 1$, such that
\eqref{eqn.temp-discr:domain-inclusion} is fulfilled.

\begin{remark}\label{remark-to-delta-assumption}
Let us comment upon, why we consider \cref{assumption.discrete-velocity}
reasonable. The discretisations of the time derivative for both the bulk and
interface velocities is a standard first-order finite difference
approximation. For sufficiently small $\dt$, we can reasonably expect that
the discrete approximation is close to the real value, so that
\eqref{eqn.assumption.discrete-velocity} can be fulfilled.
\end{remark}

Let us briefly discuss the solvability of the system 
\eqref{eqn.temp-semi:weak-form}. To this end, we introduce an iteration in
$\ubnl,\lambnl,\Ubnl,\Cnl$. Let $\ubn_0,\lambn_0,\Ubn_0,\Cn_0=\ub^{n-1},
\lamb^{n-1}, \Ub^{n-1},\C^{n-1}$, and $\O_l,\G_l$ denote the domains resulting
from the position $\Cnl$. For
$l=1,\dots$, solve 
\begin{subequations}
\begin{align}
  \frac{1}{\dt}\lprod{\ubnl}{\vb}{\O_{l-1}} 
    + \lprod{\nabla\ubnl}{\nabla\vb}{\O_{l-1}}
    + \lprod{&\lambnl}{\vb}{\G_{l-1}} + \lprod{\mub}{\ubnl}{\G_{l-1}}
    \label{eqn.itteration1}\\
  &= \lprod{\mub}{\Ubn_{l-1}}{\G_{l-1}}
    + \frac{1}{\dt}\lprod{\Ex\ub^{n-1}}{\vb}{\O_{l-1}}\nonumber\\
  \Ubnl &= \Ub^{n-1} + \dt\left(\int_{\G_{l-1}}-\partial_{\nb}\ubnl\dif s + \gb \right)
  \label{eqn.itteration2}\\
  \Cnl &= \C^{n-1} + \dt\Ubnl\label{eqn.itteration3}
\end{align}
\end{subequations}
Therefore, the system \eqref{eqn.temp-semi:weak-form}  has a solution if the
mapping $g: (\Ubnl,\Cnl)\mapsto(\Ubn_{l+1}, \Cn_{l+1})$ has a fixed point. To
this end, we observe that
\begin{equation*}
  \nrm{\Cnl - \Cn_k}{2} = \dt \nrm{\Ubnl - \Ubn_k}{2}
   = \dt^2 \Big\Vert\int_{\G_{l-1}}-\partial_{\nb}\ubnl\dif s 
                - \int_{\G_{k-1}}-\partial_{\nb}\ubn_k\dif s\Big\Vert_2.
\end{equation*}
As a result, we have that $g$ is a contraction for sufficiently small $\dt$, if the bulk solution and consequently the drag is Lipschitz continuous with respect to the interface velocity and position. This can be achieved easily by rewriting \eqref{eqn.itteration1} in a reference domain, using a smooth mapping which is a small distortion of the identity. As the solution of \eqref{eqn.itteration1} is bounded by the data, for $\Ubn_{l-1},\Cn_{l-1}$ from a bounded ball around $\Ub^{n-1},\C^{n-1}$, we have that the drag can be bounded by a uniform constant. Then for $\dt$ sufficiently small, it follows that $\Ubnl,\Cnl$ are also contained in this ball. Consequently, the Banach fixed-point theorem gives that g has a unique fixed point, so \eqref{eqn.temp-semi:weak-form} admits a unique solution.

\subsection{Stability Analysis of the Semi-Discrete Scheme}
\label{sec.temp-discr:subsec.stability}

We show that the temporal semi-discretisation \eqref{eqn.temp-semi:weak-form}
results in a stable solution.

\begin{lemma}
\label{lemma.temp-semi:stabil}
Let $\{\ub^m, \Ub^m, \C^m\}_{m=1}^N$ be the velocity and position solution to
\eqref{eqn.temp-semi:weak-form} with compatible initial data $\O(0)$ and 
$(\ub^0,\Ub^0)\in\Vb^0\times\R^d$, and the time step $\dt$ be sufficiently
small. If \Cref{assumption.discrete-velocity} holds, we have for $m=1,\dots,N$  
the stability estimate
\begin{multline*}
  \nrm{\ub^m}{\O^m}^2 + \nrm{\Ub^m}{2}^2
    + \dt\sum_{n=1}^{m}\frac{1}{2}\nrm{\nabla\ub^i}{\On}^2\\
  \leq \exp\left(t^m\frac{\clem[a]{lemma.temp-semi:stabil}}{
                          1-\dt\clem[a]{lemma.temp-semi:stabil}}\right)
    \left[
    \nrm{\ub^0}{\O^0}^2 + \nrm{\Ub^0}{2}^2 +
    \frac{\dt}{2}\nrm{\nabla\ub^0}{\O^0} +
    \clem[b]{lemma.temp-semi:stabil}\nrm{\gb}{2}^2t^m
    \right],
\end{multline*}
with constants $\clem[a]{lemma.temp-semi:stabil},
\clem[b]{lemma.temp-semi:stabil}>0$ independent of the time step and
the number of steps $n$.
\end{lemma}

\begin{proof}
We test \eqref{eqn.temp-semi:weak-form} with
$(\vb, \mub, \Vbb_1, \Vbb_2) = 2\dt(\ubn, -\lambn, \Ubn, \Cn)$ to obtain
\begin{multline}\label{eqn.temp-semi:lemma-stab.proof.1}
  2\lprod{\ubn-\Ex\ub^{i-1}}{\ubn}{\On} + 2 \lprod{\Ubn - \Ub^{n-1}}{\Ubn}{2}
    + 2 \lprod{\Cn - \C^{n-1}}{\Cn}{2}\\
    + 2\dt\nrm{\nabla\ubn}{\On} + 2\dt\lprod{\lambn}{\Ubn}{\Gn}
  = 2\dt\lprod{\gb+ \Fb^n}{\Ubn}{2} + 2 \dt \lprod{\Ubn}{\Cn}{2}.
\end{multline}
For the two terms originating from the approximation of the time-derivative, we
have the polarisation identity $ 2\lprod{\ubn-\Ex\ub^{n-1}}{\ubn}{\On} =
\nrm{\ubn}{\On}^2 + \nrm{\ubn-\Ex\ub^{n-1}}{\On}^2 -\nrm{\Ex\ub^{n-1}}{\On}^2.$
For the Lagrange-multiplier, external forcing and sold velocity-position 
coupling terms, we have as in the proof of \cref{lemma.cont-stability} above 
that
\begin{gather*}
  \lprod{\lambn}{\Ubn}{\Gn} = \lprod{\Fb^n}{\Ubn}{2}
  \intertext{and}
  \lprod{\gb}{\Ubn}{2} \leq \frac{c_1}{4}\nrm{\gb}{2}^2
    + \nrm{\nabla\ubn}{\On}^2,\qquad
  \lprod{\Ubn}{\Cn}{2} 
    \leq \frac{1}{2}\nrm{\Ubn}{2}^2 + \frac{1}{2}\nrm{\Cn}{2}^2,
\end{gather*}
with $c_{1}=c_\O^2c_P^2 / \vert\G\vert$. Using these equalities
and estimates, we get from \eqref{eqn.temp-semi:lemma-stab.proof.1} that
\begin{multline}\label{eqn.temp-semi:lemma-stab.proof.2}
  \nrm{\ubn}{\On}^2 + \nrm{\Ubn}{2}^2 + \nrm{\Cn}{2}^2 
    + \dt\nrm{\nabla\ubn}{\On}^2\\
    \leq \nrm{\Ex\ub^{n-1}}{\On}^2 + \nrm{\Ub^{n-1}}{2}^2 + \nrm{\C^{n-1}}{2}^2
      + \frac{c_{1}\dt}{2}\nrm{\gb}{2}^2 + \dt\nrm{\Ubn}{2}^2 + \dt\nrm{\Cn}{2}^2.
\end{multline}
Now, for arbitrary $\eps>0$, we have from~\cite[Lemma~3.5]{LO19} that
\begin{equation*}
  \nrm{\Ex\ub}{\Od{\O}} \leq (1 + (1 + \eps^{-1})\delta c')\nrm{\ub}{\O}^2
    + \delta c'' \eps \nrm{\nabla\ub}{\O}^2.
\end{equation*}
Then with $\delta$ as given in~\eqref{eqn.assumption.discrete-velocity} and $\eps= 1 / (2 c'' c_\delta \wninf)$, it
follows
\begin{align*}
  \nrm{\Ex\ub^{n-1}}{\On}^2
    &\leq \nrm{\Ex\ub^{n-1}}{\Od{\O^{n-1}}}^2\\
    &\leq \left(1 + \left(1 + 2c''c_\delta \wninf\right)
      c' c_\delta \wninf\dt\right)\nrm{\ub^{n-1}}{\O^{n-1}}^2
      + \frac{\dt}{2}\nrm{\nabla\ub^{n-1}}{\O^{n-1}}^2\\
    &\leq (1 + c_2\dt)\nrm{\ub^{n-1}}{\O^{n-1}}^2
     + \frac{\dt}{2}\nrm{\nabla\ub^{n-1}}{\O^{n-1}}^2.
\end{align*}
Applying this to \eqref{eqn.temp-semi:lemma-stab.proof.2} and summing this 
over $n=1,\dots,m$ leads to
\begin{multline*}
  \nrm{\ub^m}{\O^m}^2 + \nrm{\Ub^m}{2}^2 + \nrm{\C^m}{2}^2
    + \frac{\dt}{2}\sum_{n=1}^{m}\nrm{\nabla\ubn}{\On}^2
  \leq \nrm{\ub^0}{\O^0}^2 + \frac{\dt}{2}\nrm{\nabla\ub^0}{\O^0}^2
    + \nrm{\Ub^0}{2}^2\\
    + c_2\dt\sum_{n=0}^{m-1}\nrm{\ubn}{\On}^2
    + \dt\sum_{n=1}^m\left[\nrm{\Ubn}{2}^2  + \nrm{\Cn}{2}^2\right]
    + \frac{c_{1}}{2}t^m\nrm{\gb}{2}^2.
\end{multline*}
Applying a discrete version of Gronwall's lemma, see~\cite[Lemma~5.1]{HR90},
with the choices $\clem[a]{lemma.temp-semi:stabil}=\max\{c_2, 1\}$ and
$\clem[b]{lemma.temp-semi:stabil}=c_1 / 2$ then proves the claim.
\end{proof}

\section{Discretisation in Space and Time}
\label{sec.discrete}

We now come to the full discretisation of \eqref{eqn:system.strong-pde}
--\eqref{eqn:system.strong-ode}.
For the fully discrete method, we use an unfitted finite element method with
Lagrange-multipliers to implement the Dirichlet boundary condition on the
moving, unfitted interface. This will allow us to reuse aspects of the
semi-discrete analysis. This unfitted FEM has its origins in~\cite{BH10}.
The geometry is defined implicitly on a background mesh using a level set
function. So-called ``bad-cuts'' between the mesh and the boundary $\Gamma$
given by the level set function are stabilised using ghost penalty
stabilisation~\cite{Bur10}. The ghost penalty stabilisation is also
responsible for the implicit extension into a strip around the moving
interface, ensuring that the solution is well-defined on domains at subsequent
time steps. At each time step, the discrete domain $\Ohn$ is extended by a
strip of width
\begin{equation*}
  \delta_h \coloneqq c_{\delta_h}\wninf \dt
\end{equation*}
such that $\Oh^{n+1}$ is a subset of of the extended domain. We further assume
that $c_{\delta_h}>1$ is sufficiently small, such that
\begin{equation}\label{eqn.extension-domiain-inlclusion}
  \Odh{\Ohn}\subset\Od{\On}.
\end{equation}

\subsection{Spatial Discretisation Method}
\label{sec.discrete:subsec.method}

Let $\Tht$ be a simplicial, shape-regular and quasi-uniform mesh of the domain
$\Ot$, where $h>0$ is the characteristic size of the simplexes. We
collect the elements that are in the extended domain as the \emph{active mesh}
in
\begin{equation*}
  \Thnd \coloneqq \{T \in\Tht\;\vert\; \exists \xb\in T \text{ such that }
    \dist(\xb, \Ohn) \leq \delta_h \} \subset \Tht
\end{equation*}
and denote the \emph{active domain} as
\begin{equation*}
  \OdT{n} \coloneqq \{\xb\in T\;\vert\; T\in\Thnd \} \subset \R^d.
\end{equation*}
We further define the \emph{cut mesh} as $\Thn\coloneqq\mathcal{T}^{n}_{h,0}$
of all elements that contain some part of $\Ohn$ and the \emph{cut domain}
$\OTn\coloneqq\OO^n_{0,\mathcal{T}}$. Similarly, we collect the set of
\emph{interface elements} as
\begin{equation*}
  \ThGhn \coloneqq \{T\in\Tht\;\vert\; \meas_{d-1}(T\cap\Ghn) > 0 \}
\end{equation*}
and the domain of these elements as $\OGTn\coloneqq \{\xb\in T\;\vert\;
T\in\ThGhn\}.$ For the extension ghost penalty operators, we
collect the elements in the extension strip in
\begin{equation*}
  \ThnSpm \coloneqq \{T\in\Thnd \;\vert\; \exists\xb\in T \text{ such that }
    \dist(\xb,\Ghn) \leq \delta_h\}
\end{equation*}
and the set of interior facets of this strip in
\begin{equation*}
  \Fhnd \coloneqq \{ F=\overline{T}_1\cap\overline{T}_2 \;\vert\;
    T_1\in\Thnd, T_2\in\ThnSpm \text{ with }
    T_1\neq T_2\text{ and } \meas_{d-1}(F)>0\}.
\end{equation*}
Finally, for the analysis below, we also define the set of extension strip
elements 
\begin{equation*}
 \ThnSp \coloneqq \{T\in\Thnd \;\vert\; \exists\xb\in T\cap (\Ot\setminus\Ohn)
  \text{ such that } \dist(\xb,\Ghn) \leq \delta_h \}.
\end{equation*}
An illustration of these sets of elements and facets can be seen in
\cref{fig.discrete:mesh}.

\begin{figure}
  \centering
  \includegraphics[width=7.8cm]{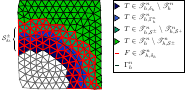}
  \caption{The different sets of element and facets needed for the discrete
    method.}
  \label{fig.discrete:mesh}
\end{figure}

\subsubsection{Finite Element Spaces}
\label{sec.discrete:subsec.method:subsubsec.fe-space}
On the active mesh, we consider for $k\geq 2$ the finite element
spaces for velocity and Lagrange-multipliers 
\begin{align*}
  \Vbhn &\coloneqq \{ \vbh \in \bm{C}(\OdT{n}) \;\vert\; \restr{\vbh}{T}\in
    \PP^k(T) \text{ for all } T\in\Thnd \},\\
  \Nbhn &\coloneqq \{ \mubh \in \bm{C}(\OGTn) \;\vert\; \restr{\mubh}{T}\in
    \PP^{k-1}(T) \text{ for all } T\in\ThGhn \}.
\end{align*}

\subsubsection{Variational Formulation}
\label{sec.discrete:subsec.method:subsubsec.form}

The fully discrete variational formulation of the method then reads as
follows: Given an appropriate and compatible set of initial conditions $\Oh^0$,
$\ub_0\in\Vbh^0$ and $\Ub_0\in\R^d$, for $n=1,\dots,N$ find $(\ubhn,\lambhn,
\Ubhn,\Chn)\in\Vbhn\times\Nbhn\times\R^d\times\R^d$ such that
\begin{multline}\label{eqn.discrete-form}
  \lprod{\frac{1}{\dt}(\ubhn - \ubh^{n-1})}{\vbh}{\Ohn} + a_h^n(\ubhn,\vbh)
  + b_h^n(\lambhn,\vbh) + b_h^n(\mubh, \ubhn-\Ubhn) \\
  + \gsu i_h^n(\ubhn,\vbh) - \gsl j_h^n(\lambhn,\mubh)
  + \lprod{\frac{1}{\dt}(\Ubhn - \Ubh^{n-1})}{\Vbb_1}{2}
  + \lprod{\frac{1}{\dt}(\Chn - \Ch^{n-1})}{\Vbb_2}{2}\\
  = \lprod{\Fbhn + \gb}{\Vbb_1}{2} + \lprod{\Ubhn}{\Vbb_2}{2}
\end{multline}
holds for all $(\vbh,\mubh,\Vbb_1,\Vbb_2)\in\Vbhn\times\Nbhn\times\R^d\times
\R^d$. The stabilisation parameters $\gsu,\gsl>0$ are to be specified later.
The bilinear forms $a_h^n(\cdot,\cdot)$ and $b_h^n(\cdot,\cdot)$ are defined by
\begin{equation*}
  a_h^n(\ubh,\vbh)\coloneqq \int_{\Ohn}\nabla\ubh\colon\nabla\vbh\dif \xb
  \qquad\text{and}\qquad
  b_h^n(\lambh,\vbh) \coloneqq \int_{\Ghn}\lambh\cdot\vbh\dif s,
\end{equation*}
respectively. To stabilise the system \eqref{eqn.discrete-form} with respect to
``bad-cuts'', we use the \emph{direct} version of the ghost penalty
stabilisation operator~\cite{Pre18}. To define this, let
$F=\overline{T}_1\cap\,\overline{T}_2$ be an interior facet and $\wF=T_1\cup T_2$
be the corresponding facet-patch. We then define $\jump{u}\coloneqq u_1 - u_2$
with $u_i = \Ex^\PP \restr{u}{T_i}$, where $\Ex^\PP\colon
\PP^m(T)\rightarrow\PP(\R^d)$ is the canonical extension of polynomials to
$\R^d$. The Laplace form is then stabilised with the ghost penalty form
\begin{equation*}
  i_h^n(\ubh,\vbh) \coloneqq \frac{1}{h^2}\!\sum_{F\in\Fhnd}
    \int_{\wF}\jump{\ubh}\cdot\jump{\vbh}\dif\xb.
\end{equation*}
Furthermore, the Lagrange-multiplier is stabilised with
\begin{equation*}
  j_h^n(\lambh,\mubh) \coloneqq 
    h^2 \int_{\OGTn}(\nb\cdot\nabla\lambh)\cdot(\nb\cdot\nabla\mubh)\dif\xb.
\end{equation*}
Here we extend the interface unit normal vector into a field in the bulk domain
by $\nb = - \nabla\phi / \nrm{\nabla\phi}{2}$. On the discrete level, we define
the force via the discrete Lagrange multiplier, i.e,
\begin{equation}\label{eqn.discrete-force}
  \Fbhn \coloneqq \int_{\Ghn}\lambhn \dif s.
\end{equation}

\subsubsection{Extension and Stabilisation through Ghost Penalties}
\label{sec.discrete:subsec.method:subsubsec.ext}

For the ghost penalty mechanism, we require the following assumption,
see also~\cite[Assumption~5.3]{LO19} and~\cite[subsection~5.1.3]{vWRL20}.

\begin{assumption}\label{assumption.ghost-penalty:strip-width}
Our analysis requires the following further assumptions on the mesh and level 
set:
\begin{enumerate}[label= - \ref*{assumption.ghost-penalty:strip-width}.\alph*]
  \item
For every element cut by $\Ghn$,
the interface $\Ghn$ intersects the element boundary $\partial T$ exactly
twice and each (open) edge exactly once.
  \item 
For each element $T$ intersected by $\Ghn$, there exists a plane $S_T$ and a
piecewise smooth parametrisation $\varphi: S_T\cap T \rightarrow \Ghn\cap T$.
  \item
We assume that for every strip element $T\in\ThnSp$ there exists an uncut
element $T' \in \Thn \setminus \ThnSp$, which can be reached by a path which
crosses a bounded number of facets $F\in\Fhnd$. We assume that the number of
facets which have to be crossed to reach $T'$ from $T$ is bounded by a constant
$K \leqc (1 + \frac{\delta_h}{h})$ and that every uncut element $T' \in \Thn
\setminus \ThnSp$ is the end of at most $M$ such paths, with $M$ bounded
independent of $\dt$ and $h$. In other words, each uncut elements "supports" at
most $M$ strip elements.
\end{enumerate}
\end{assumption}

Since the curvature of $\G$ is bounded (and remains constant in time), the 
above assumption is reasonable, if the interface is sufficiently well resolved.
\ref*{assumption.ghost-penalty:strip-width}.a and
\ref*{assumption.ghost-penalty:strip-width}.b are necessary for a
trace estimate from the interface to the entire cut element, see~\cite{HH02}.
These assumptions therefore are standard for the analysis of CutFEM methods. 
Furthermore, \ref*{assumption.ghost-penalty:strip-width}.c is standard for
unfitted moving domain discretisations. See~\cite{LO19} for a detailed
justification thereof.

We summarise the essential stabilising property of the ghost penalty operator.

\begin{lemma}\label{lemma.ghost-penalty}
With the direct ghost penalty operators, we have for all $\vbh\in\Vbhn$ that
\begin{equation*}
  \nrm{\nabla\vbh}{\OdT{n}}^2
    \eqc \nrm{\nabla\vbh}{\Ohn}^2 + K i_h^n(\vbh,\vbh).
\end{equation*}
\end{lemma}

See~\cite[Lemma~5.5]{LO19} for the proof thereof.

As seen in \cref{lemma.ghost-penalty}, the stiffness between the velocity
unknown on elements $T\in\ThnSpm\setminus\ThnSp$ and $T'\in\ThnSp$, induced by
the stabilising ghost penalty operator, depends on the inverse distance between
$T$ and $T'$ as measured in the number of elements that need to be crossed to
reach $T$ from $T'$. This in turn depends on the anisotropy between the spatial 
and the temporal discretisation, with $K \leqc (1 + \frac{\delta_h}{h})$ and
$\delta_h\leqc \dt$. In the stability analysis below, we shall require that
$\gsu\geqc K$, to compensate the weakening of the stabilisation for larger
extension strips. As a result, we choose the ghost penalty stabilisation
parameters as
\begin{align}\label{eqn.ghost-penalty-parameter}
  && \gsu = \gsu(h,\delta_h) = \gamma_s K
  \quad\text{with $\gamma_s>0$ independent of $h$ and $\dt$.}
\end{align}
See also \cite[section~4.4]{LO19}.

\subsection{Stability Analysis}
\label{sec.discrete:subsec.stability}

For our analysis, we consider the following mesh-dependent norms
\begin{align*}
  \tnrms{\ubh}^2 &\coloneqq
    \nrm{\nabla\ubh}{\OdT{n}}^2+\nrm{h^{-1/2}\ubh}{\Ghn}^2\\
  \tnrms{\lambh}^2 &\coloneqq
    \nrm{h^{1/2}\lambh}{\Ghn}^2 + \nrm{h\nb\cdot\nabla\lambh}{\OGTn}^2.
\end{align*}
Note that these norms are independent of the mesh-interface cut topology, and
since they are defined on the entire finite element spaces, they represent
proper norms on the spaces $\Vbhn$ and $\Nbhn$, respectively. On the product
space, we then take the norm
\begin{equation*}
\tnrms{(\ubh,\lambh)}^2\coloneqq \tnrms{\ubh}^2 + \tnrms{\lambh}^2.
\end{equation*}
We do not distinguish between the norms on the different spaces, since the
argument makes it clear which norm is meant.

In conjunction with the stability form, we then have the following lemma
\begin{lemma}\label{lemma.lagrange-stability.norm}
For sufficiently small $h>0$, it holds that
\begin{align*}
  && \nrm{\lambh}{\OGTn}^2 &\leqc
    \nrm{h^{1/2}\lambh}{\Ghn}^2 + \nrm{h\nb\cdot\nabla\lambh}{\OGTn}^2
  &&\text{for all }\lambh\in\Nbhn.
\end{align*}
\end{lemma}
See \cite[Sec. 7]{GLR18} for the details of the proof thereof.

\begin{lemma}
The stabilised Laplace operator $(a_h^n + \gsu i_h^n)(\cdot,\cdot)$ is
continuous and coercive on $\Vbhn$, i.e,
\begin{align*}
  && (a_h^n + \gsu i_h^n)(\ubh,\vbh) &\leqc \tnrms{\ubh}\tnrms{\vbh}
    &&\text{for all } \ubh,\vbh\in\Vbhn\\
  && (a_h^n + \gsu i_h^n)(\ubh,\ubh) &\geqc \tnrms{\ubh}^2
    &&\text{for all } \ubh\in\Vbhn.
\end{align*}
\end{lemma}
See~\cite[Lemma~6 and Lemma~7]{BH12} for a proof thereof.

\begin{lemma}
\label{lemma.bad-inf-sup}
For the discrete forms $b_h^n(\cdot,\cdot)$ and $j_h^n(\cdot,\cdot)$ in the
finite element method \eqref{eqn.discrete-form}, we have the stability estimate
\begin{align*}
  && \beta \tnrms{\lambh}
  &\leq \sup_{\vbh\in\Vbhn}\frac{b_h^n(\lambh,\vbh)}{\tnrms{\vbh}}
    + j_h^n(\lambh,\lambh)^\frac{1}{2} && \text{for all }\lambh\in\Nbhn,
\end{align*}
with the constant $\beta>0$ independent of $h$.
\end{lemma}

\begin{proof}
The proof follows ideas from~\cite[Lemma~3]{FL17}. For a given
$\lambh\in\Nbhn$, let $\lambt$ be the $\PP^{k-1}$ finite element function that
is equal to $\lambh$ in $\OGTn$ and zero in all other degrees of freedom in
$\OdT{n}$. Then we immediately have that $h\lambt\in\Vbhn$ and
\begin{equation*}
  b_h^n(\lambh, h\lambt) = \nrm{h^{1/2}\lambh}{\Ghn}^2.
\end{equation*}
By the definition of the Lagrange-multiplier norm and the Cauchy-Schwarz
inequality applied to the stabilising form, we therefore have
\begin{equation}\label{eqn.inf-sup-proof}
  \tnrms[n]{\lambh} \leq \frac{b_h^n(\lambh, h\lambt)}{\tnrms[n]{\lambh}}
    + j_h^n(\lambh,\lambh)^\frac{1}{2}.
\end{equation}

Now, let $\OdT[i]{n}$ and $\OdT[e]{n}$ denote the domain of uncut elements
inside and outside the physical domain, respectively, i.e., $\OdT{n}=
\OdT[i]{n}\,\dot\cup\,\OGTn\,\dot\cup\,\OdT[e]{n}$. We then have with the
observation that $\lambt$ is only non-zero on a strip of width $3h$, as well as
uses of the trace and inverse estimates that
\begin{align*}
  \tnrms{\lambt}^2
    &\leqc h^{-2}\nrm{\lambt}{\OdT{n}}^2 + \nrm{h^{-1/2}\lambt}{\Ghn}^2\\
    &= h^{-2}(\nrm{\lambt}{\OdT[i]{n}}^2 + \nrm{\lambt}{\OdT[e]{n}}^2
      + \nrm{\lambh}{\OGTn}^2) + \nrm{h^{-1/2}\lambh}{\Ghn}^2\\
    &\leqc h^{-1}(\nrm{\lambt}{\partial\OdT[i]{n}\cap\partial\OGTn}^2
      +\nrm{\lambt}{\partial\OdT[e]{n}\cap\partial\OGTn}^2)
      + h^{-2}\nrm{\lambh}{\OGTn}^2 + \nrm{h^{-1/2}\lambh}{\Ghn}^2\\
    &\leqc \nrm{\nabla\lambh}{\OGTn}^2\! +
      h^{-2}\nrm{\lambh}{\OGTn}^2\! + \nrm{h^{-1/2}\lambh}{\Ghn}^2\\
    &\leqc h^{-2}\nrm{\lambh}{\OGTn}^2\! + \nrm{h^{-1/2}\lambh}{\Ghn}^2.
\end{align*}
With \cref{lemma.lagrange-stability.norm}, this gives that
\begin{equation}\label{eqn.bound-lambdat-tilde}
  \tnrms{h\lambt} \leq c\tnrms{\lamb}
\end{equation}
with $c>0$ independent of $h$ and $\lambh$. Inserting this estimate on the
right-hand side of \eqref{eqn.inf-sup-proof} gives
\begin{equation*}
  \tnrms[n]{\lambh} \leq c \frac{b_h^n(\lambh, h\lambt)}{\tnrms[n]{h\lambt}}
    + j_h^n(\lambh,\lambh)^\frac{1}{2}.
\end{equation*}
The claim then follows by taking the supremum over all $\vbh\in\Vbh$.
\end{proof}

\begin{lemma}
\label{lemma.stationary-lagrange.solvability}
Let us consider the bilinear form
\begin{equation*}
A_h^{n,\ast}((\ubh,\lambh),(\vbh,\mubh))\coloneqq
(a_h^n + \gsu i_h^n)(\ubh,\vbh) + b_h^n(\lambh,\vbh) + b_h^n(\mubh,\ubh)
- j_h^n(\lambh,\mubh).
\end{equation*}
Then for all $(\ubh,\lambh)\in\Vbhn\times\Nbhn$ there holds
\begin{equation*}
  \clem{lemma.stationary-lagrange.solvability}\tnrms{(\ubh,\lambh)}\leq
  \sup_{(\vbh,\mubh)\in\Vbhn\times\Nbhn}
  \frac{A_h^{n,\ast}((\ubh,\lambh),(\vbh,\mubh))}{\tnrms{(\vbh,\mubh)}},
\end{equation*}
where the constant $\clem{lemma.stationary-lagrange.solvability}>0$ is
independent of the mesh size $h$ and the mesh-interface cut position.
\end{lemma}

\begin{proof}
Let $h\lambt\in\Vbh$ be as in the proof of \cref{lemma.bad-inf-sup}.
Then, using the coercivity and continuity of the stabilised Laplace operator,
\eqref{eqn.bound-lambdat-tilde}, \cref{lemma.ghost-penalty} and Young's
inequality we have
\begin{multline*}
  A_h^{n,\ast}((\ubhn,\lambhn), (\ubhn + \alpha h\lambt, -\lambhn))\\
  \begin{aligned}
    &\geqc \tnrms{\ubhn}^2 - \alpha\tnrms{\ubhn}\tnrms{h\lambt}
      + \alpha\nrm{h^{1/2}\lambhn}{\Ghn}^2 
      + \nrm{h\nb\cdot\nabla\lambhn}{\OGTn}^2\\
    &\geqc (1 - \frac{\alpha}{2})\tnrms{\ubhn}^2 
      + \frac{\alpha}{2}\nrm{h^{1/2}\lambhn}{\Ghn}^2
      + (1-\frac{\alpha}{2})\nrm{h\nb\cdot\nabla\lambhn}{\OGTn}^2.  
  \end{aligned}
\end{multline*}
For $\alpha$ sufficiently small, it follows that
\begin{equation*}
  A_h^{n,\ast}((\ubhn,\lambhn), (\ubhn + \alpha h\lambt, -\lambhn))
  \geqc \tnrms{(\ubhn,\lambhn)}^2.
\end{equation*}
The claim then follows due to $\tnrms{(\ubhn + \alpha h\lambt, -\lambhn)} \leqc
\tnrms{(\ubhn, \lambhn)}$, which is a consequence of the triangle inequality
and \eqref{eqn.bound-lambdat-tilde}.
\end{proof}

\begin{corollary}
The CutFEM Lagrange-Multiplier method for the stationary Poisson problem given
by $A_h^{n,\ast}((\ubhn,\lambhn),(\vbh,\mubh)) = F_h((\vbh,\mubh))$ is uniquely
solvable and the condition number of the resulting stiffness matrix is bounded
independent of the mesh interface cut position.
\end{corollary}

We now show stability of the discrete scheme in the following fully
discrete counterpart to \cref{lemma.temp-semi:stabil}.

\begin{theorem}\label{thm.discrete-stability}
Let $\{(\ubh^m,\Ubh^m, \Ch^m)\}_{m=1}^N$ be the velocity and position solution
to \eqref{eqn.discrete-form}. Then under assumptions
\ref{assumption.discrete-velocity}, \ref{assump.eulerian:extension}
and \ref{assumption.ghost-penalty:strip-width}, with $\gamma_s$
sufficiently large, and $\dt$ sufficiently small, we have for 
$m=1,\dots,N$ the stability estimate
\begin{multline*}
  \nrm{\ubh^m}{\Oh^m}^2 + \nrm{\Ubh^m}{2}^2 + \nrm{\Ch^m}{2}^2
  + \dt\sum_{n=1}^{m}\Big[\cthm[a]{thm.discrete-stability}
    \nrm{\nabla\ubhn}{\OdT{n}}^2
    + \gsl j_h^n(\lambhn,\lambhn)\Big]\\
  \leq \exp\!\left(t^m \frac{\cthm[b]{thm.discrete-stability}}{
                           1-\dt\cthm[b]{thm.discrete-stability}}\right)\!\!
  \left[
\nrm{\ubh^0}{\Oh^0}^2 + \nrm{\Ubh^0}{2}^2
  + \frac{\clem{lemma.ghost-penalty}\dt}{2}\tnrms[0]{\nabla\ubh^0}^2
  + t^m \frac{\cthm[c]{thm.discrete-stability}}{\vert\G\vert}\nrm{\gb}{2}^2
  \right]\!,
\end{multline*}
with constants $\cthm[a]{thm.discrete-stability},
\cthm[b]{thm.discrete-stability}, \cthm[c]{thm.discrete-stability} > 0$ 
independent of $h$, $\dt$ and $m=1,\dots,N$.
\end{theorem}

\begin{proof}
The proof follows similar lines to that of \cref{lemma.temp-semi:stabil}.
We test \eqref{eqn.discrete-form} with $(\vbh,\mubh,\Vbb_1,\Vbb_2)=
2\dt(\ubhn,-\lambhn, \Ubhn,\Chn)$, use BDF1 polarisation identity,
the observation that $\lprod{\lambhn}{\Ubhn}{\Ghn} = \lprod{\Fbhn}{\Ubhn}{2}$
due to \eqref{eqn.discrete-force},
and the estimates \eqref{eqn.model:cont-est:T5a} and
\eqref{eqn.model:cont-est:T6}. This leads to
\begin{multline}\label{eqn.discrete-stability1}
  \nrm{\ubhn}{\Ohn}^2 + \nrm{\Ubhn}{2}^2  + \nrm{\Chn}{2}^2 
    + \dt\nrm{\nabla\ubhn}{\Ohn}^2\\
    + 2\dt\gsu i_h^n(\ubhn,\ubhn) + 2\dt\gsl j_h^n(\lambhn,\lambhn)\\
  \leq
    \nrm{\ubh^{n-1}}{\Ohn}^2 + \nrm{\Ubh^{n-1}}{2}^2 + \nrm{\Ch^{n-1}}{2}^2
    + \dt\nrm{\Ubhn}{2}^2 + \dt\nrm{\Chn}{2}^2
    + \frac{c_1\dt}{\vert\G\vert}\nrm{\gb}{2}^2.
\end{multline}
To deal with the norm of $\ubh^{n-1}$ on $\Ohn$, we recall
\cite[Lemma~5.7]{LO19}, i.e., there exists a constant $c_{\text{LO}5.7}>0$,
independent of $\dt$ and $h$, such that
\begin{multline*}
  \nrm{\ubh}{\Odh{\Ohn}}^2
  \leq (1 + c_{\text{LO}5.7a}(\eps))\nrm{\ubh}{\Ohn}^2
   + c_{\text{LO}5.7b}(\eps)\dt\nrm{\nabla\ubh}{\Ohn}\\
   + c_{\text{LO}5.7c}(\eps, h)\dt K i_h^n(\ubh,\ubh)
\end{multline*}
with constants $c_{\text{LO}5.7a}(\eps) = c_{\text{LO}5.7}c_{\delta_h}\wninf(1 +
\eps^{-1})$, $c_{\text{LO}5.7b}(\eps) = c_{\text{LO}5.7}c_{\delta_h}
\wninf\eps$ and $c_{\text{LO}5.7c}(\eps, h) = c_{\text{LO}5.7}c_{\delta_h}
\wninf(\eps + h^2 + h^2\eps^{-1})$.
Since we have that $\Ohn\subset\Odh{\Oh^{n-1}}$, we have with the choice of
$\eps \leq 1 / 2 c_{\text{LO}5.7}c_{\delta_h}\wninf$ that
\begin{multline}\label{eqn.est-previous-vel-on-domain.discr}
  \nrm{\ubh^{n-1}}{\Ohn}^2 %
  \leq (1 + c_2\dt)\nrm{\ubh^{n-1}}{\Oh^{n-1}}^2
  + \frac{1}{2}\dt\nrm{\nabla\ubh^{n-1}}{\Oh^{n-1}}^2\\
  + c_3 K \dt i_h^{n-1}(\ubh^{n-1},\ubh^{n-1}),
\end{multline}
where $c_2,c_3>0$ are independent of $\dt$ and $h$. Inserting this into
\eqref{eqn.discrete-stability1}, under the assumption that $\gsu\geq c_3 K$,
summing over $n=1,\dots,k\leq N$ and applying \cref{lemma.ghost-penalty}
then gives
\begin{multline*}
  \nrm{\ubh^k}{\Oh^k}^2 + \nrm{\Ubh^k}{2}^2 + \nrm{\Ch^k}{2}^2
  + \dt\sum_{n=1}^{k}\Big[c_4\nrm{\nabla\ubhn}{\OdT{n}}^2
    + \gsl j_h^n(\lambhn,\lambhn)\Big]\\
  \leq \nrm{\ubh^0}{\Oh^0}^2 + \nrm{\Ubh^0}{2}^2
  + \frac{\clem{lemma.ghost-penalty}\dt}{2}\tnrms[0]{\nabla\ubh^0}^2
  + t^k \frac{c_1}{\vert\G\vert}\nrm{\gb}{2}^2\\
  + \dt c_2\sum_{n=1}^{k-1} \nrm{\ubhn}{\Ohn}^2
  + \dt \sum_{n=1}^k\left[\nrm{\Ubhn}{2}^2 + \nrm{\Chn}{2}^2\right].
\end{multline*}
The claim then follows by an application of the discrete form of Gronwall's
lemma.
\end{proof}

\subsection{Domain Error}
\label{sec.discrete:subsec.doimain-err}

In this section, we shall formalise the discrepancy between the exact domain
and the domain resulting from the discretised problem.
In contrast to~\cite{LO19,vWRL20,LL21} but as in~\cite{BFM19}, we shall assume
exact geometry handling, i.e., if the motion of the domain were known, we would
have $\O(t^n)=\Ohn$. We thus choose to ignore the geometry consistency
error of order $\OO(h^2)$ introduced by the piecewise linear level set
approximation inherent in CutFEM. We do this to focus our analysis on the
fact, that in the discretised setting, the motion of the domain results from
a discretised from of the ODE governing the motion. As a result there is a
miss-match between the motion of the domain between the smooth and discrete
case, i.e., $\O(t^n)\neq\Ohn$, and our analysis will focus on this error
source.

We note that the correct geometry order for high-order finite element spaces
can be recovered in CutFEM by using, for example, the isoparametric CutFEM
approach~\cite{Leh16}, which has been studied for a range of stationary
problems and recently extended to moving domain problems in~\cite{LL21}.

Now, the position of $\Sigma(t^n)$ and $\Sigma_h^n$ are governed by
\begin{equation*}
  \Ctn = \int_0^{t^n}\Ub\dif t = \C(t^{n-1})
    + \int_{t^{n-1}}^{t^n}\Ub\dif t
  \qquad\text{and}\qquad
  \Chn = \Ch^{n-1} + \dt\Ubhn.
\end{equation*}
The difference $\Ctn-\Chn$ represent the miss-match between the domains at
time $t^n$, which results from the discretisation of the problem
\eqref{eqn.model:weak-form}. To analyse this error, we define a mapping from
the discrete domain to the exact domain by $\Phi^n\colon \Ohn 
\rightarrow \O(t^n)$ by
\begin{equation}\label{eqn.domain:mapping}
  \Phi^n \coloneqq \id + (\Ctn - \Chn)\varphi^n,
\end{equation}
where $\varphi^n\in C^\infty(\Ohn)$, such that $\restr{\varphi^n}{\Ghn}=1$ and
$\restr{\varphi^n}{\Ohn\setminus\OO_{+}(\Ghn)}=0$. We take this mapping to be
invertible. In the following, this mapping takes a similar role as the geometry
approximation mapping in, e.g.,~\cite{LO19,vWRL20}. Consequently, the domain 
error can be quantified by $\nrm{\id - \Phi^n}{}$ in an appropriate norm.

\begin{lemma}\label{lemma.deform-mapping}
Let $\nrm{\cdot}{\infty}$ be the $\pazocal{L}^\infty$ norm on $\Ohn$. Then for
the mapping $\Phi^n:\Ohn\rightarrow\O(t^n)$ defined in
\eqref{eqn.domain:mapping}, it holds that
\begin{equation}\label{eqn.definition-M}
  \nrm{\id - \Phi^n}{\infty}^2 \leqc \dt^2\sum_{j=0}^n \nrm{\EU[j]}{2}^2
    + \dt^3t^n\nrm{\partial_t \Ub }{\infty}^2 \eqqcolon \MF(\EU[n], \dt),
\end{equation}
where $\EU[n]\coloneqq \Ub(t^n) - \Ubhn$.
\end{lemma}
\begin{proof}
By the definitions of $\Phi^n$, $\Ctn$ and $\Chn$, and the fact that
$\varphi^n$ is smooth, we have that
\begin{align*}
  \nrm{\id - \Phi^n}{\infty}^2
    \leqc \nrm{\Ctn - \Chn}{2}^2 \nrm{\varphi^n}{\infty}^2
    &\leqc \nrm{\Ctn - \Chn}{2}^2\\
    &\leqc \textstyle \nrm{\C(t^{n-1}) - \C_h^{n-1}}{2}^2
      + \nrm{\int_{t^{n-1}}^{t^{k}}\Ub(t) - \Ubhn\dif t}{2}^2.
\end{align*}
With respect to the final term, we have for $t\in[t^{n-1}, t^n]$ that
\begin{equation*}
  \Ub(t) \leq \Ub(t^n) - t \nrm{\partial_t \Ub}{\infty, [t^{n-1}, t^n]}
    \leq \Ub(t^n) - \dt\nrm{\partial_t \Ub }{\infty, [t^{n-1}, t^n]}.
\end{equation*}
Therefore, we have the bound
\begin{align*}
  \Big\Vert\int_{t^{n-1}}^{t^{k}}\Ub(t) - \Ubhn\dif t\Big\Vert_2^2
   &\leq \Big\Vert\int_{t^{n-1}}^{t^{k}}\Ub(t^n)
      - \dt\nrm{\partial_t \Ub }{\infty, [t^{n-1}, t^n]}
      - \Ub^n\dif t\Big\Vert_{2}^2\\
   &\leq \dt^2 \nrm{\EU[n]}{2}^2
    + \dt^4\nrm{\partial_t \Ub }{\infty,[t^{n-1},t^n]}^2.
\end{align*}
Iteratively repeating the above estimate for the
$\nrm{C(t^{n-1}) - C_h^{n-1}}{2}^2$ term then gives
\begin{align*}
  \nrm{\id - \Phi^n}{\infty}^2
    &\leqc \sum_{j=0}^n \left[\dt^2 \nrm{\EU[j]}{2}^2
    + \dt^4\nrm{\partial_t \Ub }{\infty,[t^{j-1},t^j]}^2\right]\\
    &\leqc \dt^2\sum_{j=0}^n \nrm{\EU[j]}{2}^2
      + \dt^3t^n\nrm{\partial_t \Ub }{\infty}^2.
\end{align*}
\end{proof}

\begin{lemma}\label{lemma.deform.est:eqiv}
For $\Phi^n$ defined in \eqref{eqn.domain:mapping}, describing the
mismatch between the exact and the discrete domain at time $t^n$, we have that
\begin{align*}
  \nrm{I - D\Phi^n}{\infty} &\simeq \nrm{\id - \Phi^n}{\infty}, &
  \nrm{1 - \det(D\Phi^n)}{\infty}&\simeq \nrm{\id - \Phi^n}{\infty},\\
  \nrm{\id - \Phi^n}{\infty,\Gn}&\simeq \nrm{\id - \Phi^n}{\infty}, &
  \nrm{1 - \det(D\Phi^n)}{\infty,\Gn}&\simeq \nrm{\id - \Phi^n}{\infty}.
\end{align*}
\end{lemma}
\begin{proof}
This follows by $\varphi^n\in C^\infty(\On)$ and the fact that the remaining
components of $\Phi^n$ are independent of space.
\end{proof}

\begin{lemma}\label{lemma.deform-est}
Let $\MF(\EU[n], \dt)$ be as defined in \eqref{eqn.definition-M}.
Then for $\ub\in\Hb{3}{}{\O(t^n)}$, it holds that
\begin{align}
  \nrm{\ub\cphin - \Ex\ub}{\Ohn}^2 &\leqc \nrm{\ub}{\Hb{1}{}{\O(t^n)}}^2
    \MF(\EU[n], \dt),\label{eqn.defoem-ext-est:l2}\\
  \nrm{(\nabla\ub)\cphin - \nabla\Ex\ub}{\Ohn}^2
    &\leqc \nrm{\ub}{\Hb{2}{}{\O(t^n)}}^2
    \MF(\EU[n], \dt),\label{eqn.defoem-ext-est:h1}\\
  \nrm{\ub\cphin - \Ex\ub}{\Ghn}^2
    &\leqc \nrm{\ub}{\Hb{2}{}{\O(t^n)}}^2
   \MF(\EU[n], \dt),\label{eqn.defoem-ext-est:trace}\\
  \nrm{(\partial_{\nb}\ub)\cphin - \partial_{\nb}\Ex\ub}{\Ghn}^2
    &\leqc \nrm{\ub}{\Hb{3}{}{\O(t^n)}}^2
    \MF(\EU[n], \dt).\label{eqn.defoem-ext-est:normal-deriv}
\end{align}
\end{lemma}

\begin{proof}
$\Phi^n$ maps the approximated interface location $\Ghn$ to the exact interface
location $\G(t^n)$, and we know that the distance between the two is given by
$\vert \C(t^n) - \Chn\vert$ for which we have proven the estimate in the proof
of \cref{lemma.deform-mapping}. The proof of
\eqref{eqn.defoem-ext-est:l2}--\eqref{eqn.defoem-ext-est:normal-deriv} is
therefore completely analogous to that of the geometry approximation error
in~\cite[Lemma~7.3]{GOR15}.
\end{proof}

We note that the domain error still depends on the error of the interface
velocity. This is to be expected, since we will only be able to bound this
error together with the entire velocity error.

\subsection{Consistency Error}
\label{sec.discrete:subsec.consistency-err}

In this section, we analyse the consistency of our discrete formulation.
To ease the upcoming notation, we shall identify $\ub$ and $\lamb$ with their
extensions. Due to \eqref{eqn.extension-domiain-inlclusion}, we can define the
error on the discrete domain. Therefore, let us define the bulk-velocity,
interface-velocity and Lagrange-multiplier errors as
\begin{equation}\label{eqn.definition-errors}
  \begin{aligned}
    \Eu[n]&\coloneqq \ub(t^n) - \ubhn,\quad&
    \El[n]&\coloneqq \lamb(t^n) - \lambhn,\\
    \EU[n]&\coloneqq \Ub(t^n) - \Ubhn,&
    \EC[n]&\coloneqq \C(t^n) - \Chn 
\end{aligned} 
\end{equation}
Now, to derive an error equation for our discretisation, we observe that if
$(\vbh,\mubh)\in\Vbhn\times\Nbhn$ are suitable
test-functions for the discrete problem \eqref{eqn.discrete-form},
then they are not necessarily valid test-functions for the smooth problem
\eqref{eqn.model:weak-form}. However, using the mapping $\Phi^n$, we define
$\vbhl\coloneqq\vbh\cphininv$ and $\mubhl\coloneqq\mubh\cphininv$. Inserting
these test-function into \eqref{eqn.model:weak-form}, subtracting 
\eqref{eqn.discrete-form}, as well as adding and subtracting appropriate
terms, we get the error equation
\begin{multline}\label{eqn.consistecy-error}
  \frac{1}{\dt}\lprod{\Eu[n] - \Eu[n-1]}{\vbh}{\Ohn}^2
    + \lprod{\nabla\Eu[n]}{\nabla\vbh}{\Ohn}^2
    + \gsu i_h^n(\Eu[n], \vbh)
    + \lprod{\El[n]}{\vbh}{\Ghn}\\
  + \lprod{\mubh}{\Eu[n] - \EU[n]}{\Ghn}
    - \gsl j_h^n(\El[n], \mubh)
    + \frac{1}{\dt}\lprod{\EU[n] - \EU[n-1]}{\Vbb_1}{2}
    + \lprod{\Fbhn - \Fb(t^n)}{\Vbb_1}{2}\\
    + \frac{1}{\dt}\lprod{\EC[n] - \EC[n-1]}{\Vbb_2}{2}
    - \lprod{\EU[n]}{\Vbb_2}{2}
  = \Ecn(\vbh, \mubh, \Vbb_1, \Vbb_2),
\end{multline}
with the consistency error
\begin{multline*}
  \Ecn(\vbh, \mubh, \Vbb_1, \Vbb_2)\coloneqq
     \underbrace{\frac{1}{\dt}\lprod{\ub(t^n) - \ub(t^{n-1})}{\vbh}{\Ohn}
      - \lprod{\partial_t\ub(t^n)}{\vbhl}{\O(t^n)}}_{\TT_1}\\
  + \underbrace{ \lprod{\nabla\ub(t^n)}{\nabla\vbh}{\Ohn}
      - \lprod{\nabla\ub(t^n)}{\nabla\vbhl}{\O(t^n)}}_{\TT_2}
    + \underbrace{ \lprod{\lamb(t^n)}{\vbh}{\Ghn}
      - \lprod{\lamb(t^n)}{\vbhl}{\G(t^n)}}_{\TT_3}\\
  + \underbrace{ \lprod{\mubh}{\ub(t^n) - \Ub(t^n)}{\Ghn}
      - \lprod{\mubhl}{\ub(t^n) - \Ub(t^n)}{\G(t^n)}}_{\TT_4}\\
  + \underbrace{ \frac{1}{\dt}\lprod{\Ub(t^n) - \Ub(t^{n-1})}{\Vbb_1}{2}
      - \lprod{\frac{\dif}{\dif t}\Ub}{\Vbb_1}{2}}_{\TT_5}
    + \underbrace{ \frac{1}{\dt}\lprod{\C(t^n) - \C(t^{n-1})}{\Vbb_2}{2}
      - \lprod{\frac{\dif}{\dif t}\C}{\Vbb_2}{2}}_{\TT_6}\\
  + \underbrace{ \gsu i_h^n(\ub(t^n),\vbh)}_{\TT_7}
      - \underbrace{\gsl j_h^n(\lamb(t^n),\mubh)}_{\TT_8}.
\end{multline*}

\begin{lemma}[Consistency Estimate]
\label{lemma.consistency-error-est}
Let the $\ub$ fulfil the regularity assumption
$\ub\in\Wb{3,\infty}{\QQ}\cap\pazocal{L}^{\infty}
(0,\tend;\Hb{k+1}{}{\O(t)})$, then the consistency error can be bounded by
\begin{multline*}
  \vert \Ecn \vert \leqc
    (\dt + h^k K^{\frac{1}{2}} + \MF(\EU[n], \dt)^\frac{1}{2})\Big(
      \nrm{\ub}{\Wb{3,\infty}{\QQ}}
      + \sup_{t\in[0,\tend]}\nrm{\ub}{\Hb{k+1}{}{\O(t)}}\Big)\tnrms[n]{\vbh}\\
       + \dt\sup_{t\in[0,\tend]}\nrm{\Ub_{tt}}{2}\nrm{\Vbb_1}{2}
       + \dt\sup_{t\in[0,\tend]}\nrm{\Ub_{t}}{2}\nrm{\Vbb_2}{2}.
\end{multline*}

\end{lemma}

\begin{proof}
The proof follows similar lines as~\cite[Lemma~5.11]{LO19}.

For the time derivative term we have with a change of variable, that
\begin{align*}
  \vert \TT_1 \vert
  &= \Big\vert -\int_{\Ohn}\int_{t^{n-1}}^{t^n}\frac{t-t^n}{\dt}
      \ub_{tt}(t)\dif t \cdot \vbh\dif \xb + \lprod{\ub_t(t^n)}{v^n}{\Ohn}
    - \lprod{\ub_t(t^n)}{v^n_\ell}{\O(t^n)}\Big\vert\\
  &\leq \frac{1}{2}\dt\nrm{\ub}{\Wb{2,\infty}{\QQ}}\nrm{\vbh}{\Ohn}
    + \vert \lprod{\ub_t(t^n) - \ub_t(t^n)\cphin}{\vbh}{\Ohn}
    \vert\\
  &\leqc (\dt + \MF(\EU[n], \dt)^\frac{1}{2})
    \nrm{\ub}{\Wb{2,\infty}{\QQ}}\nrm{\vbh}{\Ohn}.
\end{align*}
In the final step, we have used
\begin{equation*}
  \vert \ub_t(x,t^n) - (\ub_t\cphin)(x,t^n)\vert
  \leq \nrm{\nabla\ub_t}{\Lb{\infty}{}{\Od{\O(t^n)}}}\vert x - \Phi^n(x)\vert
\end{equation*}
and \cref{lemma.deform-mapping}. See also~\cite[Lemma~5.11]{LO19}.

For the diffusion term $\TT_2$, it follows analogously from the
differentiation chain rule, \cref{lemma.deform.est:eqiv} and
\cref{lemma.deform-est} that, see, e.g.~\cite[Lemma~7.4]{GOR15}
\begin{align*}
  \vert \TT_2 \vert
  &= \begin{multlined}[t]
    \vert\lprod{\nabla(\ub(t^n)-\ub(t^n)\cphin)}{J(D\Phi^n)^{-1})\nabla\vbh}{\Ohn}\\
    + \lprod{\nabla\ub(t^n)}{(I-J(D\Phi^n)^{-1})\nabla\vbh}{\Ohn}\vert
  \end{multlined}\\
  &\leq \MF(\EU[n],\dt)^\frac{1}{2}\nrm{\ub(t^n)}{\Hb{2}{}{\O(t^n)}}
    \nrm{\nabla\vbh}{\Ohn}.
\end{align*}

For the first Lagrange-multiplier term, we similarly find by additionally using
a trace and the Poincaré inequality that
\begin{align*}
  \vert \TT_3 \vert
  &= \vert \lprod{\lamb(t^n)- \lamb(t^n)\cphin}{J\vbh}{\Ghn}
    - \lprod{\lamb(t^n)}{(1-J)\vbh}{\Ghn} \vert\\
  &\leqc \MF(\EU[n], \dt)^\frac{1}{2}\nrm{\ub(t^n)}{\Hb{3}{}{\O(t^n)}}
    \nrm{\nabla\vbh}{\Ohn}.
\end{align*}

For the boundary condition term, we first note that due to $\Ub\in\R^d$, we can
identify the extension as the constant extension. Furthermore, we can choose
the extension of the bulk velocity, such that $\Ex\ub = \Ub$ in the
$\delta_h$-strip around $\G(t^n)$ and $\Ex\ub=\ub$ outside of a
$2\cdot\delta_h$-strip around $\G(t^n)$, with $\Ex\ub$ sufficiently smooth in
$\Odh{\Ohn}$. As a result, we have that
\begin{align*}
  \vert \TT_4 \vert
  &= \vert \lprod{J\mubh}{\ub(t^n)\cphin - \Ub(t^n)\cphin}{\Ghn}
    - \lprod{\mubh}{\ub(t^n) - \Ub(t^n)}{\Ghn}\vert = 0.
\end{align*}

The interface-velocity consistency error is bounded similar to $\TT_1$.
However, the situation is simpler here because $\Ub\in\R^d$ does not
depend on the domain consistency. Therefore,
\begin{align*}
  \vert \TT_5 \vert
    = \Big\vert \frac{1}{\dt}\lprod{\Ub(t^n) - \Ub(t^{n-1})}{\Vbb_1}{2}
    - \lprod{\frac{\dif}{\dif t} \Ub(t^n)}{\Vbb_1}{2}\Big\vert
    \leqc \dt\nrm{\frac{\dif{}^2}{\dif t^2} \Ub}{\infty, [t^{n-1}, t^n]}
      \nrm{\Vbb_1}{2}.
\end{align*}

Similarly for the interface position, we with $\frac{\dif}{\dif t}\C = \Ub$
that
\begin{align*}
  \vert \TT_6 \vert
    = \Big\vert \frac{1}{\dt}\lprod{\C(t^n) - \C(t^{n-1})}{\Vbb_2}{2}
    - \lprod{\frac{\dif}{\dif t} \C(t^n)}{\Vbb_2}{2}\Big\vert
    \leqc \dt\nrm{\frac{\dif}{\dif t} \Ub}{\infty, [t^{n-1}, t^n]}
      \nrm{\Vbb_2}{2}.
\end{align*}

Finally, for the ghost penalty consistency error, we use for
$\ub\in\Hb{k+1}{}{\OdT{n}}$ the consistency estimate,
see~\cite[Lemma~5.8]{LO19},
\begin{equation}\label{eqn.ghost-penalty-consistency}
  i_h^n(\ub,\ub)\leqc h^{2k}\nrm{\ub}{\Hb{k+1}{}{\OdT{n}}}^2.
\end{equation}
As a result, we have using the Cauchy-Schwarz inequality
\begin{align*}
 \vert \TT_7 \vert
   \leqc i_h^n(\ub,\ub)^{\frac{1}{2}}i_h^n(\vbh,\vbh)^{\frac{1}{2}}
   &\leqc h^k\nrm{\ub}{\Hb{k+1}{}{\OdT{n}}}i_h^n(\vbh,\vbh)^{\frac{1}{2}}\\
   &\leqc h^k\nrm{\ub}{\Hb{k+1}{}{\O(t^n)}}i_h^n(\vbh,\vbh)^{\frac{1}{2}}
\end{align*}
where we used \eqref{eqn.eulerian:ExtentionExtimate} in the last estimate.

To the Lagrange-multiplier stabilisation form, we use that we identify 
$\lamb$ in the bulk with the function which is equal to $\lamb$ in the
interface and which is constant in the normal direction $\nb$. Thus the 
stabilisation is fully consistent and vanishes.

\end{proof}

\subsection{Error estimate in the energy norm}
\label{sec.discrete:subsec.err-est}

We consider stable interpolation operators $\Iu,\Il$ for the bulk velocity
and the Lagrange-multiplier spaces $\Vbh,\Nbn$, respectively. For 
$k_s=1,\dots,k$, $\ub\in\Hb{k_s+1}{}{\O(t^n)}$ and
$\lamb\in\Hb{k_s-1/2}{}{\G(t^n)}$, it then holds that
\begin{multline}
    \nrm{\widetilde\ub - \Iu\widetilde\ub}{\OdT{n}}
    + h\nrm{\nabla(\widetilde\ub - \Iu\widetilde\ub)}{\OdT{n}}
    + h^{\frac{1}{2}}\nrm{\widetilde\ub - \Iu\widetilde\ub}{\Ghn}\\
    \leqc{}h^{k_s + 1}\nrm{\ub}{\Hb{k_s+1}{}{\O(t^n)}}
    \label{eqn.interpolation:velocity}
\end{multline}
\begin{multline}
  \nrm{\widetilde\lamb - \Il\widetilde\lamb}{\OGTn} 
    + h\nrm{\nabla(\widetilde\lamb - \Il\widetilde\lamb)}{\OGTn}
    + h^{\frac{1}{2}}\nrm{\widetilde\lamb - \Il\widetilde\lamb}{\Ghn}\\
    \leqc h^{k_s}\nrm{\lamb}{\Hb{k_s - 1/2}{}{\G(t^n)}},
   \label{eqn.interpolation:lagrange}
\end{multline}
with sufficiently smooth extensions $\widetilde\ub\in\Hb{k+1}{}{\OdT{n}}$ 
and $ \widetilde\lamb\in\Hb{k}{}{\OGTn}$ for which it holds that
$\restr{\widetilde\ub}{\O(t^n)} = \ub$ and $\widetilde\lamb \big|_{\Ghn} 
= \lamb$. The existence of $\widetilde{\lamb}$ is given by
the trace theorem and we have $\nrm{\widetilde\lamb}{\Hb{k}{}{\OGTn}}\leqc
\nrm{\lamb}{\Hb{k-1/2}{}{\Ghn}}$, see~\cite{FL17} for further details, and
note that the necessary assumptions are given by the assumptions of our mesh
and the smoothness of the level set function. As in the previous section, we
can for example identify $\lamb$ in the bulk with the function which is equal
to $\lamb$ in the interface and which is constant in the normal direction
$\nb$.

Let $\ubIn\coloneqq \Iu\ub$ and $\lambIn\coloneqq \Il\lamb$.
We split the bulk velocity and the Lagrange-multiplier errors into an
interpolation and a discretisation error
\begin{equation*}
  \Eu[n] =  \underbrace{(\ub(t^n) - \ubIn)}_{\errIu[n]}
          + \underbrace{(\ubIn - \ubhn)}_{\erru[n]\in\Vbhn}
  \qquad\text{and}\qquad
  \El[n] =  \underbrace{(\lamb(t^n) - \lambIn)}_{\errIl[n]}
          + \underbrace{(\lambIn - \lambhn)}_{\errl[n]\in\Nbhn}.
\end{equation*}
Note that we do not need to split the surface velocity error, since
$\EU[n]\in\R^d$ is already finite dimensional. Applying this split in
\eqref{eqn.consistecy-error} yields
\begin{multline}\label{eqn.discretisation-error}
  \frac{1}{\dt}\lprod{\erru[n] - \erru[n-1]}{\vbh}{\Ohn}^2
    + \lprod{\nabla\erru[n]}{\nabla\vbh}{\Ohn}^2
    + \gsu i_h^n(\erru[n], \vbh)
    + \lprod{\errl[n]}{\vbh}{\Ghn}\\
    + \lprod{\mubh}{\erru[n] - \EU[n]}{\Ghn}
    - \gsl j_h^n(\errl[n], \mubh)
    + \frac{1}{\dt}\lprod{\EU[n] - \EU[n-1]}{\Vbb_1}{2}
    + \lprod{\Fbhn - \Fb(t^n)}{\Vbb_1}{2}\\
  + \frac{1}{\dt}\lprod{\EC[n] - \EC[n-1]}{\Vbb_2}{2}
    - \lprod{\EU[n]}{\Vbb_2}{2}
    = \Ecn(\vbh, \mubh, \Vbb_1, \Vbb_2) + \EIn(\vbh, \mubh),
\end{multline}
for all $(\vbh,\mubh,\Vbb_1,\Vbb_2 )\in\Vbhn\times\Nbhn\times\R^d\times\R^d$,
with the interpolation term
\begin{equation*}
  \EIn(\vbh, \mubh) =
  \begin{multlined}[t]
    - \underbrace{\frac{1}{\dt}\lprod{\errIu[n] 
                                      - \errIu[n-1]}{\vbh}{\Ohn}^2}_{\TT_9}
    - \underbrace{\lprod{\nabla\errIu[n]}{\nabla\vbh}{\Ohn}^2}_{\TT_{10}}
    - \underbrace{\gsu i_h^n(\errIu[n], \vbh)}_{\TT_{11}}\\
    - \underbrace{\lprod{\errIl[n]}{\vbh}{\Ghn}}_{\TT_{12}}
    - \underbrace{\lprod{\mubh}{\errIu[n]}{\Ghn}}_{\TT_{13}}
    + \underbrace{\gsl j_h^n(\errIl[n], \mubh)}_{\TT_{14}}.
  \end{multlined}
\end{equation*}

\begin{lemma}[Interpolation estimate]
\label{lemma.interpolation-error-est}
Let $\ub\in\pazocal{L}^\infty(0,\tend,\Hb{k+1}{}{\O(t)})$ and $\ub_t\in
\pazocal{L}^\infty(0,\tend,\Hb{k}{}{\O(t)})$. Then the interpolation error
can be bounded by
\begin{align*}
  \vert \EIn(\vbh,\mubh)\vert \leqc h^k K^\frac{1}{2}\!\sup_{t\in[0,\tend]}\!
    \Big( \nrm{\ub}{\Hb{k+1}{}{\O(t)}} + \nrm{\ub_t}{\Hb{k}{}{\O(t)}}\Big)
    \tnrms[n]{\vbh}&\\
    + h^k \!\sup_{t\in[0,\tend]}\nrm{\ub}{\Hb{k+1}{}{\O(t)}}\tnrms[n]{\mubh}&.
\end{align*}
\end{lemma}

\begin{proof}
The bound
\begin{equation*}
  \vert \TT_9 + \TT_{10} + \TT_{11} \vert \leqc h^k K^\frac{1}{2}\!
    \sup_{t\in[0,\tend]}\!\Big( \nrm{\ub}{\Hb{k+1}{}{\O(t)}} 
    + \nrm{\ub_t}{\Hb{k}{}{\O(t)}}\Big)\tnrms[n]{\vbh}
\end{equation*}
is shown in~\cite[Lemma~5.12]{LO19}. We therefore only need to deal with the
boundary terms.

Using the Cauchy-Schwarz inequality, and since the Lagrange-multiplier
is the normal derivative of the velocity, we find with the stability of the 
extension that
\begin{align*}
  \vert\TT_{12}\vert 
    \leq h^{\frac{1}{2}}\nrm{\errIl[n]}{\Ghn} h^{-\frac{1}{2}}\nrm{\vbh}{\Ghn}
    &\leqc h^k\nrm{\partial_n\ub}{\Hb{k-1/2}{}{\G(t^n)}}
      h^{-\frac{1}{2}}\nrm{\vbh}{\Ghn}\\
    &\leqc h^k\nrm{\ub}{\Hb{k+1}{}{\O(t^n)}}\tnrms[n]{\vbh}.
\end{align*}
Similarly, we have
\begin{align*}
  \vert\TT_{13}\vert
    \leq h^{-\frac{1}{2}}\nrm{\errIu[n]}{\Ghn} h^{\frac{1}{2}}\nrm{\mubh}{\Ghn}
    &\leqc h^k\nrm{\ub}{\Hb{k+1}{}{\OdT{n}}} \nrm{h^{\frac{1}{2}}\mubh}{\Ghn}\\
    &\leqc h^k\nrm{\ub}{\Hb{k+1}{}{\O(t^n)}}\tnrms[n]{\mubh}.
\end{align*}
For the Lagrange-multiplier stabilising form, we again use the Cauchy-Schwarz
inequality and the interpolation estimate \eqref{eqn.interpolation:lagrange}.
This results in
\begin{equation*}
  \vert \TT_{14} \vert 
    \leq j_h^n(\errIl[n],\errIl[n])^\frac{1}{2}j_h^n(\mubh,\mubh)^\frac{1}{2}
    \leqc h^k\nrm{\ub}{\Hb{k+1}{}{\O(t^n)}}\nrm{h\nb\cdot\nabla\mubh}{\OGTn}.
\end{equation*}
The claim then follows by the triangle inequality and summing up the above 
estimates.
\end{proof}

\begin{theorem}[Energy error estimate]
\label{theorem.error}
Let $\{(\ubh^m,\Ubh^m)\}_{m=1}^N$ be the solution to the discrete problem
\eqref{eqn.discrete-form}. We assume that assumptions
\ref{assumption.discrete-velocity}, \ref{assump.eulerian:extension} and
\ref{assumption.ghost-penalty:strip-width} hold, assume $\gamma_s$ in
\eqref{eqn.ghost-penalty-parameter} is sufficiently large, the time step
$\dt$ is sufficiently small and the exact solution fulfils the regularity 
$\ub\in \Wb{3,\infty}{\QQ}\cap\Lb{\infty}{}{0,\tend,\Hb{k+1}{}{\O(t)}}$, 
$\ub_t\in \Lb{\infty}{}{0,\tend,\Hb{k}{}{\O(t)}}$ and 
$\Ub_{tt}\in\Lb{\infty}{}{0,\tend,\R^d}$. Then for $m=1,\dots,N$, and the 
errors defined in \eqref{eqn.definition-errors}, the following error estimate
holds:
\begin{multline*}
  \nrm{\Eu[m]}{\Oh^m}^2 + \nrm{\EU[m]}{2}^2 + \nrm{\EC[m]}{2}^2
    + \frac{\dt \cthm[a]{theorem.error}}{2} 
    \sum_{n=1}^m\Big[\tnrms[m]{\Eu[n]}^2 + j_h^n(\El[n],\El[n])\Big]\\
  \leqc
    \exp\left(\frac{\cthm[b]{theorem.error}}{1-\dt\cthm[b]{theorem.error}}
    t^m \right)(\dt^2 + h^{2k}K + h^{2k-1} + h^{2k}\dt^{-1})R(\ub,\Ub),
\end{multline*}
with $R(\ub,\Ub) \coloneqq \nrm{\ub}{\Wb{3,\infty}{\QQ}}^2 
+ \sup_{t\in[0,\tend]} \Big( \nrm{\ub}{\Hb{k+1}{}{\O(t)}}^2 
+ \nrm{\ub_t}{\Hb{k}{}{\O(t)}}^2 + \nrm{\Ub_{tt}}{2}^2 
+ \vert\G\vert\Big[\nrm{\ub}{\Hb{3}{}{\O(t)}}^2 
+ \nrm{\ub}{\Hb{k+1}{}{\O(t)}}^2 \Big] \Big)$, and constants
$\cthm[a]{theorem.error}, \cthm[b]{theorem.error} >0$ independent of $n,\dt$
and the mesh-interface cut topology.
\end{theorem}

\begin{proof}
Testing \eqref{eqn.discretisation-error} with $(\vbh, \mubh, \Vbb_1, \Vbb_2) =
2\dt(\erru[n], -\errl[n], \EU[n], \EC[n])$ gives
\begin{multline}\label{eqn:err-eqn.tested}
  \nrm{\erru[n]}{\Ohn}^2 + \nrm{\erru[n] - \erru[n-1]}{\Ohn}^2
  - \nrm{\erru[n-1]}{\Ohn}^2 + 2\dt\nrm{\nabla\erru[n]}{\Ohn}^2
  + 2\dt\gsu i_h^n(\erru[n], \erru[n])\\
  + 2\dt \lprod{\errl[n]}{\EU[n]}{\Ghn} + 2\dt\gsl j_h^n(\errl[n],\errl[n])
  + \nrm{\EU[n]}{2}^2 + \nrm{\EU[n] - \EU[n-1]}{2}^2 - \nrm{\EU[n-1]}{2}^2\\
  + 2\dt\lprod{\Fbhn - \Fb(t^n)}{\EU[n]}{2}
  + \nrm{\EC[n]}{2}^2 + \nrm{\EC[n] - \EC[n-1]}{2}^2 - \nrm{\EC[n-1]}{2}^2
  - 2\dt \lprod{\EU[n]}{\EC[n]}{2}\\
  = 2\dt (\Ecn(\erru[n], -\errl[n], \EU[n], \EC[n]) + \EIn(\erru[n], -\errl[n])).
\end{multline}
Now, by definition, we have that $\Fbhn = \int_{\Ghn}\lambh\dif s$ and
$\Fb(t^n) = \int_{\G(t^n)}\lamb(t^n)\dif s$. Since $\EU[n]\in\R^d$ is constant
in space, we have
\begin{equation*}
  \lprod{\Fbhn - \Fb(t^n)}{\EU[n]}{2}
    = \lprod{\lambhn}{\EU[n]}{\Ghn} - \lprod{\lamb(t^n)}{\EU[n]}{\G(t^n)}.
\end{equation*}
As a result, the boundary integrals involving $\lambh$ on the left-hand side of
\eqref{eqn:err-eqn.tested} vanish, and we have an additional mixed
consistency/interpolation error term
\begin{equation*}
  \Ern(\EU[n]) \coloneqq
  \lprod{\lamb(t^n)}{\EU[n]}{\G(t^n)} - \lprod{\lambIn}{\EU[n]}{\Ghn}
\end{equation*}
on the right-hand side. This leads to the error equation
\begin{multline}\label{eqn:err-eqn.tested2}
  \nrm{\erru[n]}{\Ohn}^2 + \nrm{\erru[n] - \erru[n-1]}{\Ohn}^2
  - \nrm{\erru[n-1]}{\Ohn}^2 + 2\dt\nrm{\nabla\erru[n]}{\Ohn}^2\\
  + 2\dt\gsu i_h^n(\erru[n], \erru[n])
  + 2\dt\gsl j_h^n(\errl[n],\errl[n])\\
  + \nrm{\EU[n]}{2}^2 + \nrm{\EU[n] - \EU[n-1]}{2}^2 - \nrm{\EU[n-1]}{2}^2
  + \nrm{\EC[n]}{2}^2 + \nrm{\EC[n] - \EC[n-1]}{2}^2 - \nrm{\EC[n-1]}{2}^2\\
  = 2\dt [\Ecn(\erru[n], -\errl[n], \EU[n])
          + \EIn(\erru[n], -\errl[n]) + \Ern(\EU[n]) + \lprod{\EU[n]}{\EC[n]}{2}].
\end{multline}
We begin by deriving an estimate for the additional error term $\Ern(\EU[n])$.
For this, we split the error term into
\begin{equation*}
  \Ern(\EU[n]) = \lprod{\lamb - \lambI}{\EU[n]}{\G(t^n)}
    + \lprod{\lambI}{\EU[n]}{\G(t^n)} -\lprod{\lambI}{\EU[n]}{\Ghn}.
\end{equation*}
For the first term, we estimate using the Cauchy-Schwarz inequality, the fact 
that $\EU[n]\in\R^d$ and the interpolation estimate
\eqref{eqn.interpolation:lagrange}, that
\begin{align*}
  \lprod{\lamb - \lambI}{\EU[n]}{\G(t^n)}
  \leq \nrm{\errIl[n]}{\Ghn} \nrm{\EU[n]}{\Ghn}
  &\leqc h^{k-\frac{1}{2}}\nrm{\ub}{\Hb{k+1}{}{\O(t^n)}}
    \vert\G\vert^\frac{1}{2}\nrm{\EU[n]}{2}.
\end{align*}
Note that we loose half an order in $h$ here, by considering the interface
velocity error. This would be recovered, by the appropriate $h$-scaling of the
boundary term in the bulk-velocity norm. 

For the second term, we again use that $\EU[n]\in\R^d$, so that $\EU[n] = 
\EU[n]\cphininv$. As a result, we can use a change of variable,
\cref{lemma.deform-est} and the boundedness of the interpolation operator to
estimate
\begin{align*}
  \vert \lprod{\lambI}{\EU[n]}{\G(t^n)} -\lprod{\lambI}{\EU[n]}{\Ghn} \vert
  &= \lprod{\lambIn}{\EU[n]\cphininv}{\G(t^n)} - \lprod{\lambIn}{\EU[n]}{\Ghn}\\
  &= \lprod{\lambIn\cphin  - \lambIn}{J\EU[n]}{\Ghn}
    + \lprod{\lambIn(J-1)}{\EU[n]}{\Ghn}\\
  &\leqc \MF(\EU[n], \dt)^\frac{1}{2} \nrm{\ub}{\Hb{3}{}{\O(t^n)}}
      \vert\G\vert^\frac{1}{2}\nrm{\EU[n]}{2}.
\end{align*}
Using these estimates, together with \cref{lemma.ghost-penalty},
\cref{lemma.consistency-error-est} and
\cref{lemma.interpolation-error-est} in \eqref{eqn:err-eqn.tested2} gives
\begin{multline}\label{eqn:err-eqn.tested3}
  \nrm{\erru[n]}{\Ohn}^2 + \nrm{\erru[n] - \erru[n-1]}{\Ohn}^2
    + 2c_1\dt\tnrms[n]{\erru[n]}^2 + 2\gsl j_h^n(\errl[n],\errl[n])
    + \nrm{\EU[n]}{2}^2 + \nrm{\EC[n]}{2}^2\\
  \leq 
    \nrm{\erru[n-1]}{\Ohn}^2 + \nrm{\EU[n-1]}{2}^2 + \dt\nrm{\EU[n]}{2}^2
    + \dt\nrm{\EC[n]}{2}^2 + 2c\dt\TT_{15}
\end{multline}
with
\begin{multline*}
  \TT_{15} \coloneqq
    (\dt + h^kK^{\frac{1}{2}} + \MF^\frac{1}{2})R_1(\ub)\tnrms[n]{\erru[n]}
    +  h^k R_2(\ub)\tnrms[n]{\errl[n]}\\
    +  (\dt + h^{k-\frac{1}{2}} + \MF^\frac{1}{2})R_3(\ub,\Ub)\nrm{\EU[n]}{2}.
\end{multline*}
Here we have abbreviated $\MF = \MF(\EU[n], \dt)$ and the higher-order residual
terms are
\begin{align*}
  R_1(\ub) &\coloneqq \nrm{\ub}{\Wb{3,\infty}{\QQ}} + \sup_{t\in[0,\tend]}
    \Big( \nrm{\ub}{\Hb{k+1}{}{\O(t)}} + \nrm{\ub_t}{\Hb{k}{}{\O(t)}}\Big),\\
  R_2(\ub) &\coloneqq \sup_{t\in[0,\tend]}\nrm{\ub}{\Hb{k+1}{}{\O(t)}}\\
  R_3(\ub, \Ub) &\coloneqq \sup_{t\in[0,\tend]}\Big(
    \nrm{\Ub_{tt}}{2} + \nrm{\Ub_{t}}{2}
    + \vert\G\vert^\frac{1}{2} \big[\nrm{\ub}{\Hb{3}{}{\O(t)}}
    + \nrm{\ub}{\Hb{k+1}{}{\O(t)}} \big]
    \Big).
\end{align*}
As in \eqref{eqn.est-previous-vel-on-domain.discr} we estimate the first term
on the right-hand side of \eqref{eqn:err-eqn.tested3} by
\begin{equation*}
  \nrm{\erru[n-1]}{\Ohn}^2 \leq (1 + c'\dt)\nrm{\erru[n-1]}{\Oh^{n-1}}^2
  + \frac{1}{2}c_1\dt\tnrms[n-1]{\erru[n-1]}^2.
\end{equation*}
For the Lagrange-multiplier, we test \eqref{eqn.discretisation-error} with
$(\vbh,\mubh,\Vbb_1,\Vbb_2) = (\vbh,0,0,0)$ and use \cref{lemma.bad-inf-sup}
to get the estimate
\begin{equation*}
  \tnrms[n]{\errl[n]} \leqc 
    \frac{1}{\dt} \nrm{\erru[n] - \erru[n-1]}{\Ohn} + \tnrms{\erru[n]}
    + (\dt + h^k K^\frac{1}{2} + \MF^\frac{1}{2})R_1(\ub)
    + j_h^n(\errl[n], \errl[n])^\frac{1}{2}.
\end{equation*}
Then, we can estimate using the weighted Young's inequality
\begin{multline*}
  2c\dt \TT_{15} \leq
    \frac{1}{2}\nrm{\erru[n] - \erru[n-1]}{\Ohn}^2
      + c_1\dt\tnrms[n]{\erru[n]}^2
      + \dt\gsl j_h^n(\lambhn,\lambhn)
      + \dt\nrm{\EU[n]}{2}^2\\
    + \tilde{c}(2 + \dt R_2^2)\MF
      + c'' \dt^2 R(\ub,\Ub) (\tnrms[n]{\erru[n]}^2 + \nrm{\EU[n]}{2}^2)\\
    + c_4 \dt (\dt^2 + h^{2k}K + h^{2k-1} + h^{2k}\dt^{-1})R(\ub,\Ub).
\end{multline*}
Inserting these estimates in \eqref{eqn:err-eqn.tested3} then gives
\begin{multline}\label{eqn:err-eqn.tested4}
  \nrm{\erru[n]}{\Ohn}^2 \myplus \frac{1}{2}\nrm{\erru[n] - \erru[n-1]}{\Ohn}^2
    \myplus c_1\dt\tnrms[n]{\erru[n]}^2 \myplus \dt\gsl j_h^n(\errl[n], \errl[n])
    \myplus \nrm{\EU[n]}{2}^2 \myplus \nrm{\EC[n]}{2}^2\\
  \leq 
    (1 + c'\dt)\nrm{\erru[n-1]}{\Oh^{n-1}} 
    + \frac{c_1\dt}{2}\tnrms[n-1]{\erru[n-1]}
    + \nrm{\EU[n-1]}{2}^2 + \nrm{\EC[n-1]}{2}^2\\
    + 2 \dt \nrm{\EU[n]}{2}^2 + \dt \nrm{\EC[n]}{2}^2
    + c'' \dt^2 R(\ub,\Ub) (\tnrms[n]{\erru[n]}^2 + \nrm{\EU[n]}{2}^2)\\
    + c'''(2 + \dt R_2^2)\MF
    + c_4 \dt (\dt^2 + h^{2k}K + h^{2k-1} + h^{2k}\dt^{-1})R(\ub,\Ub).
\end{multline}

Now, the goal is to sum this over $n=1,\dots,m$ in order to use a discrete 
Gronwall lemma. To this end, we first note that
\begin{align*}
  \sum_{n=1}^m\MF(\EU[n],\dt) 
  &= \sum_{n=1}^{m}\bigg[
    \dt^2\sum_{j=0}^n \nrm{\EU[j]}{2}^2+\dt^3t^n\nrm{\partial_t \Ub }{\infty}^2
    \bigg]\\
  &\leq t^m \dt \sum_{n=1}^m\nrm{\EU[n]}{2}^2 
    + \dt^2(t^m)^2\nrm{\Ub_t}{\infty}^2.
\end{align*}
Summing \eqref{eqn:err-eqn.tested4} over $n=1,\dots,m$, and using that
$\erru[0]=\errl[0]=\EU[0]=\EC[0]=0$, then gives
\begin{multline*}
  \nrm{\erru[m]}{\Oh^m}^2 + \nrm{\EU[m]}{2}^2 + \nrm{\EC[m]}{2}^2
    + \dt\big(\frac{c_1}{2} - c'''\dt R(\ub,\Ub)\big)
      \sum_{n=1}^{m}\tnrms[n]{\erru[n]}^2 
    + \dt \sum_{n=1}^{m} \gsl j_h^n(\errl[n], \errl[n])\\
  \leq c'\dt\sum_{n=1}^{m-1}\nrm{\erru[n]}{\Ohn}^2
    + \dt \big[2 + t^m c''(2 + \dt R_2^2) + c'''\dt R(\ub,\Ub)\big]
    \sum_{i=1}^{m}\nrm{\EU[n]}{2}^2\\
  + \dt\sum_{n=1}^{m}\nrm{\EC[n]}{2}^2
    + c_4 (\dt^2 + h^{2k}K + h^{2k-1} + h^{2k}\dt^{-1})R(\ub,\Ub),
\end{multline*}
with $c_4>0$ independent of $\dt$ and $m$. We now take $\dt$ to be
sufficiently small, such that $(c_1/2 - c'''\dt R(\ub,\Ub) \geq c_1/4$, and
$\dt (2 + t^m c''(2 + \dt R_2^2) + c'''\dt R(\ub,\Ub)) 
\eqqcolon \dt\cthm[b]{theorem.error} < 1$.
Under this time step restriction, we can then apply a
discrete Gronwall's Lemma~\cite[Lemma~5.1]{HR90} to get the estimate
\begin{multline*}
  \nrm{\erru[m]}{\Oh^m}^2 + \nrm{\EU[m]}{2}^2 + \nrm{\EC[m]}{2}^2
    + \dt \cthm[a]{theorem.error} 
    \sum_{n=1}^{m}\Big[\tnrms[n]{\erru[n]}^2 + j_h^n(\lambhn,\lambhn)\Big]\\
  \leqc 
    \exp\left(\frac{\cthm[b]{theorem.error}}{1-\dt\cthm[b]{theorem.error}} 
    t^m\right)(\dt^2 + h^{2k}K + h^{2k-1} + h^{2k}\dt^{-1})R(\ub,\Ub),
\end{multline*}
with $\cthm[a]{theorem.error}\coloneqq\max\{c_1/4, \gsl\}$. The claim then
follows by the triangle inequality and the optimal interpolation properties. 
\end{proof}

We note that the $h^{2k}/\dt$ scaling also appears in the error estimate in
\cite{vWRL20} for the transient Stokes problem on a moving domain, also as a
result of the use of an inf-sup result, in the latter case for the pressure.

\section{Numerical Examples}
\label{sec.num-examples}
We have implemented the method using \texttt{ngsxfem} \cite{LHPvW21}, an
add-on to \texttt{NGSolve}/\texttt{netgen}~\cite{Sch14,Sch97} for unfitted
finite element discretisations. The reproduction source code can be found in
the archive~\cite{vWR21a_zenodo}.

\subsection{Set-up}
\label{sec.num-examples:subsec.set-up}
We consider the background domain $\Ot = (0,1)^2$, and the initial domain of
interest is given by $\O(0) = \Ot \setminus \{\xb\in\Ot\;\vert\;
(\xb_1 - 0.5)^2 + (\xb_2 - 0.8^2)\le  0.1^2 \}$. The external force acting on
$\Sigma$ is given by $\gb = (0, -1)^T$. At $t=0$, the system is at rest, i.e.,
$\ub=\bm{0}$ and $\Ub_0=\bm{0}$. The system is considered until $t=\tend=1.5$.

As we do not have an analytical solution for this problem, we shall compare
our results against a reference solution. As quantities of interest
for comparison with this reference simulation, we consider the
position and velocity of the moving interface. The error is then
measured in the discrete space-time norm
\begin{equation*}
  \nrm{\Ub^\text{ref} - \Ubh}{\ell^2(\R^d)}^2
  \coloneqq \sum_{i=1}^{N} \dt\nrm{\Ub^\text{ref}(t^i) - \Ubh^i}{2}^2.
\end{equation*}

\begin{remark}[Reference Simulation]
\label{remark.reference-simulation}
To compute a reference simulation of the above set-up, we consider a fitted
ALE discretisation. Here we use $\PP^4$ elements together with BDF2
time stepping. As the motion of the domain is purely translational, a simple
analytical form of the ALE mapping can be given, see, e.g.,~\cite{vWRFH21}. The
PDE/ODE system is solved using a partitioned approach as in the Eulerian
setting, c.f.~\cref{remark.implementational-details}. An illustration of the
solution can be seen in \autoref{fig.reference-sol}.
\end{remark}

\begin{figure}
  \centering
  \includegraphics[width=3.1cm]{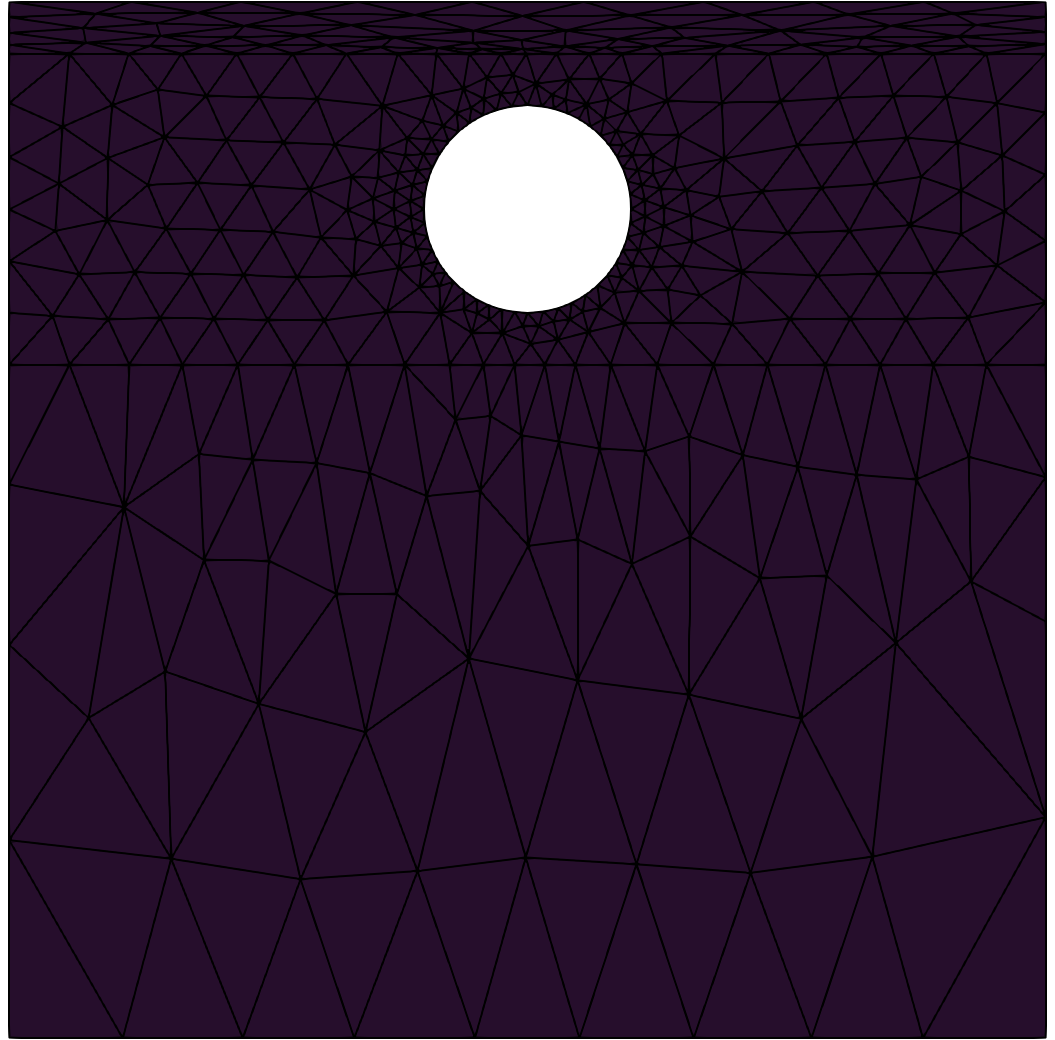}
  \includegraphics[width=3.1cm]{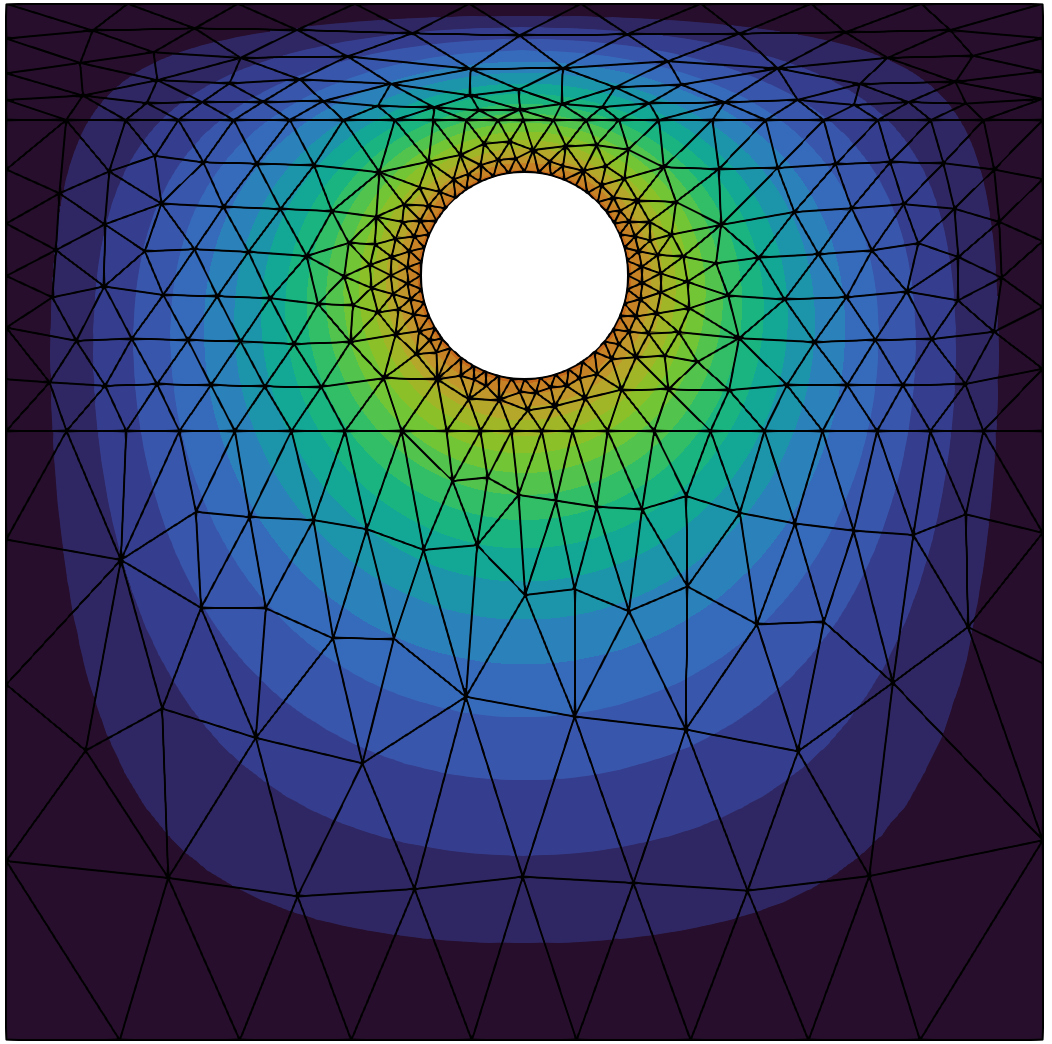}
  \includegraphics[width=3.1cm]{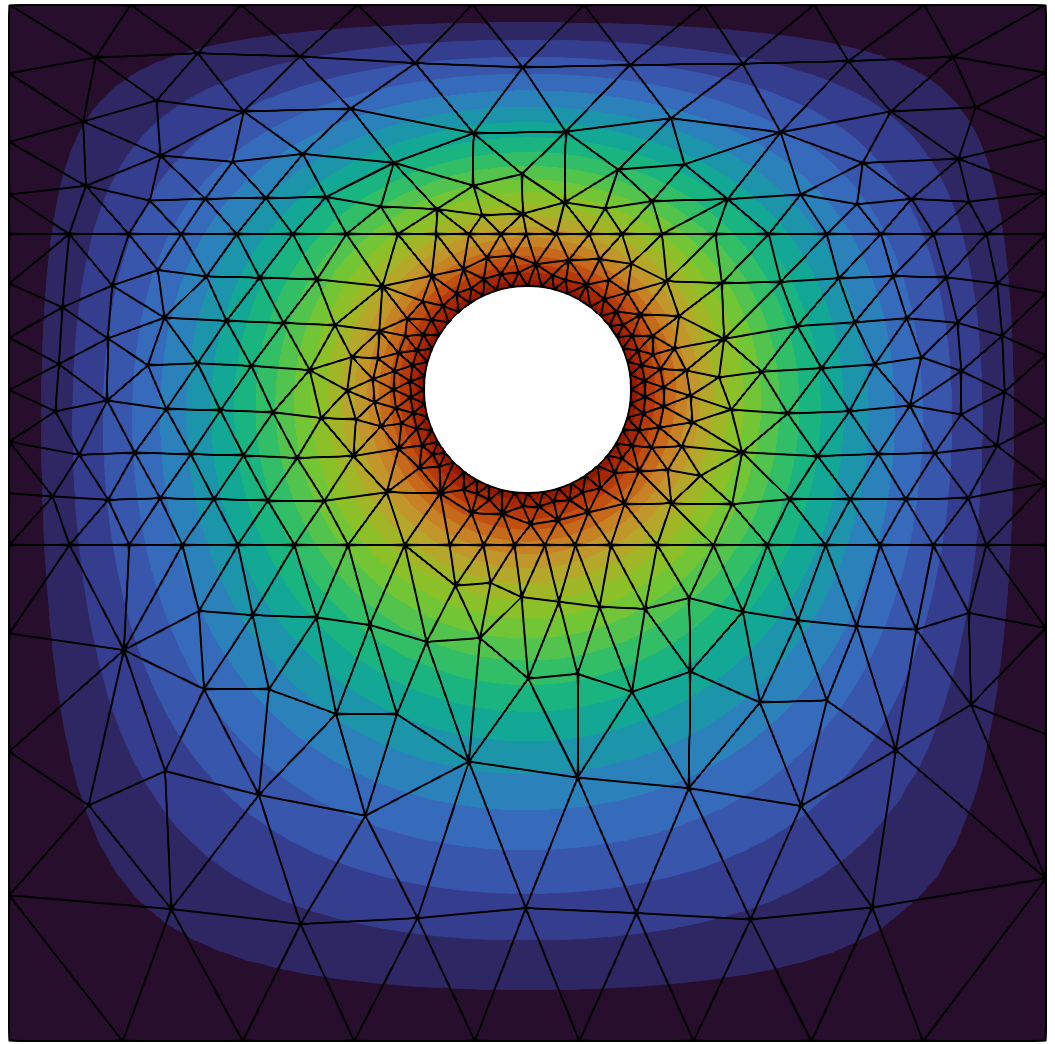}
  \includegraphics[width=3.1cm]{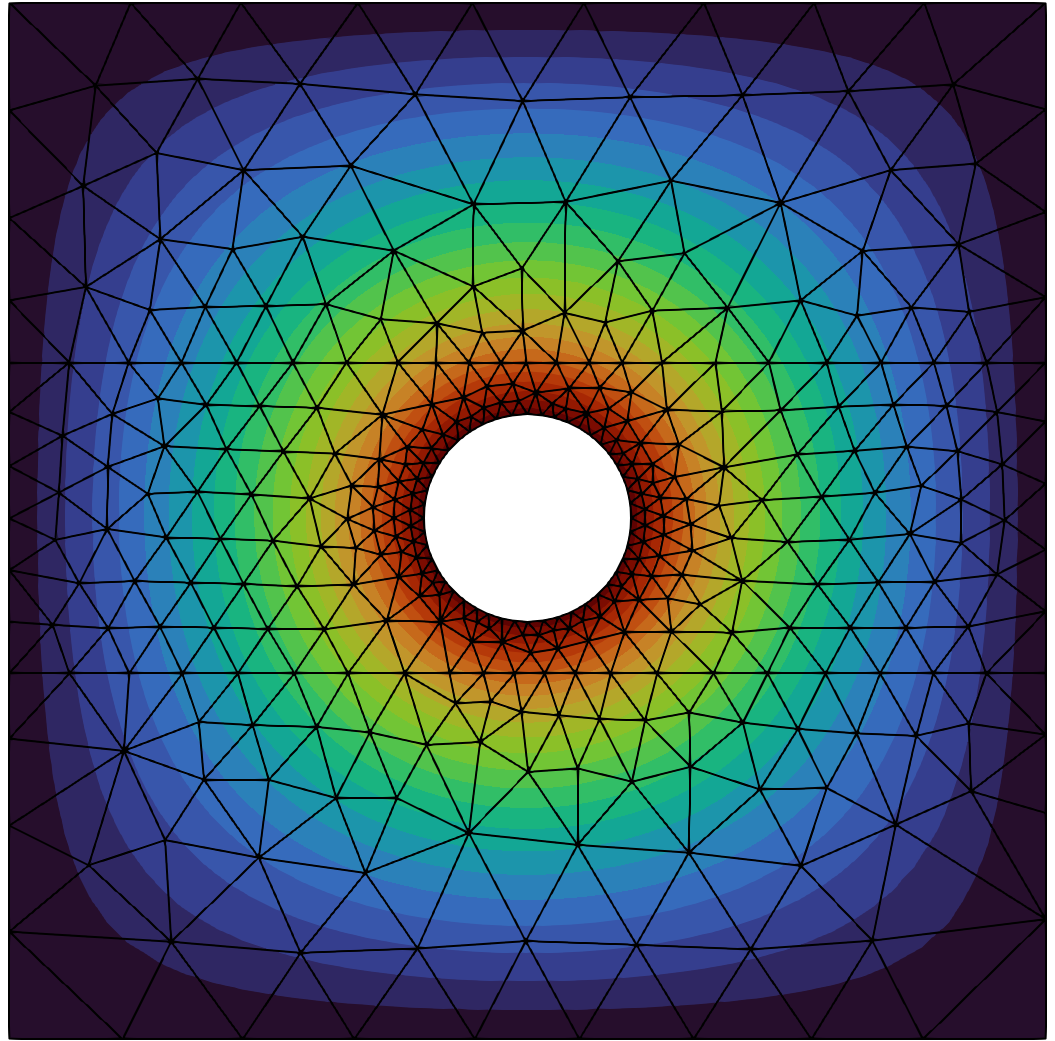}\\
  \vspace{5pt}
  \includegraphics[width=7cm]{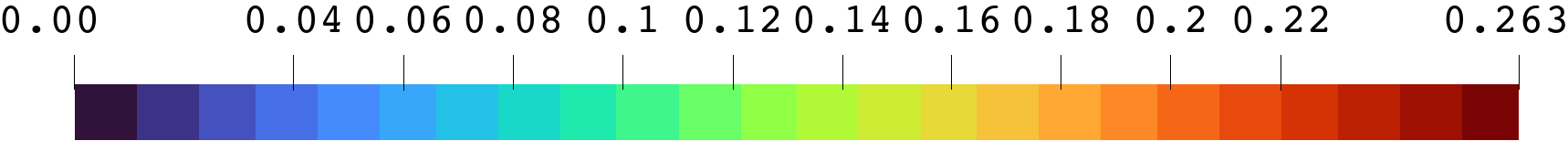}
  \caption{Solution magnitude at $t=0, 0.5, 1, 1.5$. Computed using the ALE
    method with $\hmax=0.1$, $k=3$ and $\dt=\frac{1}{200}$.}
  \label{fig.reference-sol}
\end{figure}

\begin{remark}
\label{remark.implementational-details}
To solve the coupled PDE/ODE system, we use a partitioned approach with a
relaxation in the update of the interface velocity for stability of the scheme.
To increase the convergence of the relaxation scheme to update the interface
velocity, we use Aitken's $\Delta^2$-method~\cite{KW08} to determine a good
value for the relaxation parameter. In practice, we then require three
iterations between the ODE and PDE until the relative velocity update is less
than $10^{-8}$.
\end{remark}

\subsection{Convergence Study}
\label{sec.num-examples:subsec.conv}
We consider a series of shape-regular and quasi uniform meshes of the
background domain $\Ot$, with the mesh sizes $\hmax=h_0 \cdot 2^{-L_x}$.
The initial mesh size is taken to be $h_0=0.1$. On these meshes we consider
$\PP^2/\PP^1$ elements for the velocity and Lagrange-multiplier spaces.
Similarly for the time step, we consider a series of uniformly refined time steps
of $\dt = \dt_0 \cdot 2^{-L_t}$, with the initial time step $\dt_0 =
\frac{1}{50}$. We then consider the spatial convergence using the smallest
time step and the temporal convergence on the finest mesh. The results for the
full study over each mesh/time step combination can be found in the
archive~\cite{vWR21a_zenodo}.

To define the extension strip, we set $\wninf =\vert \Ubhn\cdot\nb\vert$ in each
time step, i.e., the explicit normal interface velocity. To ensure that the
strip is wide enough to allow for acceleration of the interface, we set
$c_{\delta_h}=2$. Furthermore, the ghost penalty parameter is set to
$\gamma_s=0.1$ and the Lagrange-multiplier stabilisation parameter is set to
$\gsl=0.01$.

\paragraph{BDF1}
We consider the BDF1 time-discretisation for the bulk and interface velocities
as analysed above. The integration over the level set domains is then performed
using standard CutFEM, i.e., using a $\PP^1$ approximation of the level set to
explicitly construct the unfitted quadrature rules. As a result, we can only
expect spatial convergence of order two, due to the geometry error of order
$\OO(h^2)$. However, due to the first order time-discretisation, we expect the
temporal error to dominate most situations.

The errors resulting from the convergence study for the interface velocity
and position can be seen in \cref{fig.results:bdf1}. Here we see the expected
linear convergence with respect to the time step. Concerning the spatial 
convergence, we see the expected second order convergence until the temporal
discretisation error starts to dominate.

\begin{figure}
  \centering
  \includegraphics{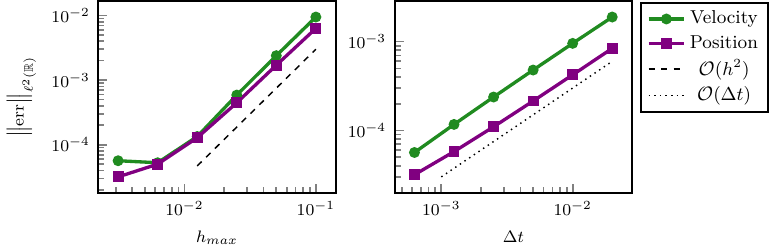}
  \caption{Error convergence for the interface velocity and position with 
    respect to the mesh size and the time step for the BDF1 scheme, $\PP^2$ 
    elements and a piecewise linear level set approximation.}
  \label{fig.results:bdf1}
\end{figure}

\paragraph{BDF2 and isoparametric mapping}
As an extension, we present some numerical results, based on a BDF2
discretisation of the time-derivatives. This is enabled, by making the
extension strip sufficiently large, such that both
$\Ohn,\Oh^{n-1}\subset\Odh{\Oh^{n-2}}$ and $\Ohn\subset\Odh{\Oh^{n-1}}$.
We achieve this by setting $c_{\delta_h}=4$.
Furthermore, to increase the geometry approximation properties of the CutFEM
method, we apply the isoparametric mapping approach introduced in~\cite{Leh16}
to increase the geometry approximation to $\OO(h^{k+1})$. For details of the
higher-order discretisation in space and time, applied to a moving domain
problem with prescribed motion, we refer to~\cite{LL21}.

The results over the same series of meshes and time steps as considered before,
can be seen in \cref{fig.results:bdf2ho}. Here, we see second-order
convergence with respect to the time step over the entire series of considered
time steps. With respect to the spatial discretisation, we observe forth-order
convergence until the temporal error begins to dominate. This suggests that the
error in the velocity and position of $\Sigma$ is dominated by the error in the
force acting from $\O$ onto $\G$, because the force is obtained with the 
accuracy of order $\OO(h^{2k})$, when computing this from the Lagrange
multiplier. The proof of this follows the same lines as for the
Babu\v{s}ka-Miller trick, see for example~\cite{BR06}.
Furthermore, since the temporal error appears to remain dominant on finer
meshes, we only see higher-order convergence over the first meshes.
\begin{figure}
  \centering
  \includegraphics{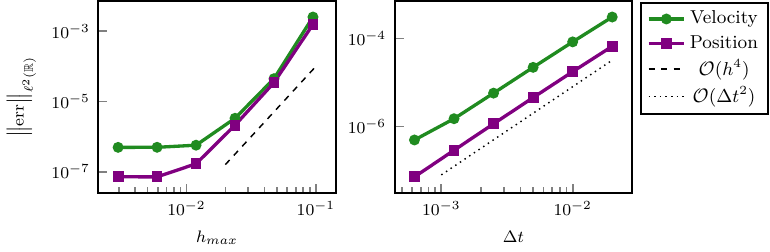}
  \caption{Error convergence for the interface velocity and position with 
    respect to the mesh size and the time step for the BDF2 scheme with $\PP^2$ 
    elements and an isoparametric CutFEM.}
  \label{fig.results:bdf2ho}
\end{figure}

\subsection{Lagrange-Multipliers vs. Nitsche}
\label{sec.num-examples:subsec.lag-vs-nit}
Our analysis relies heavily on the Lagrange-multiplier formulation to enforce
the boundary condition on the moving interface since testing with appropriate
functions then lead to some terms appearing twice with differing signs and thus
vanishing. However, CutFEM using Nitsche's method~\cite{Nit71} to implement the
boundary condition is much more commonly used, see amongst others~\cite{AB21,
BFM19, MLLR15, vWRL20}. An advantage of the Nitsche approach is that we do not
need to discretise the Lagrange-multiplier space, the resulting systems will be
smaller on the same mesh. On the other hand, we need to choose the
stabilisation parameter.

To investigate whether there is a significant numerical difference between the
two methods, we consider the BDF2 implementation together with the isoparametric
mapping and implement the boundary condition on the moving interface using
Nitsche's method. We use the symmetric version of Nitsche's method and the
penalty parameter is chosen as $40 k^2 / h$.

The results can be seen in \cref{fig.results:bdf2honitsche}. The results
are very similar compared to the Lagrange-multiplier results above. We keep
the second-order convergence in time and the higher-order convergence in space,
before the temporal error begins to dominate.

\begin{figure}
  \centering
  \includegraphics{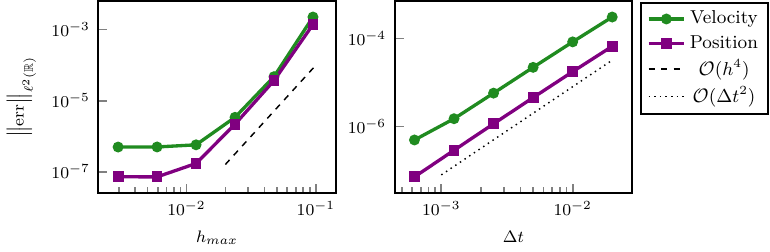}
  \caption{Error convergence for the interface velocity and position with 
    respect to the mesh size and the time step for the BDF2 scheme with $\PP^2$ 
    elements, Nitsche's method to enforce the boundary condition on the 
    moving interface and an isoparametric CutFEM.}
  \label{fig.results:bdf2honitsche}
\end{figure}

\section{Conclusions and Outlook}
\label{sec.conclusions}

In this work, we have studied an Eulerian, unfitted finite element method for
a model moving domain problem, consisting of a parabolic PDE in the bulk
together with translational coupled rigid body motion determining the motion
of the domain. The analysis of the method relied on the Lagrange multiplier
formulation we considered to implement the implicit non-homogeneous Dirichlet
boundary condition on the moving interface.
We showed stability in the temporal semi-discrete case in
\cref{lemma.temp-semi:stabil} and the stability of the fully discrete scheme
in \cref{thm.discrete-stability}. For the error analysis, we treated the
domain error resulting from the discretised scheme similarly as
the geometry approximation error is dealt with in unfitted finite element
analysis for problems with stationary domains or domains with known motion.
With this, we proved an optimal-in-time error estimate for the bulk- and
interface-velocity in the energy norm with \cref{theorem.error}. This
estimate takes a similar form as the case for prescribed motion, with the
key differences being a higher regularity assumption on the exact solution and
a stronger time step restriction, resulting from the use of a more general form
of Gronwall's lemma. 
We illustrated our theoretical results with numerical convergence studies in 
both space and time. For the time-derivative approximation, we observed optimal
order convergence for both the BDF1 and BDF2 time-derivative approximation.
Concerning the mesh size, we observed the expected second-order
convergence in the case of a piecewise linear level set approximation, while
we observed higher-order convergence when an isoparametric CutFEM approach was
used for improved geometry approximation. Furthermore, we observed that in
practice, there is no noticeable difference in the accuracy of the analysed
Lagrange-multiplier formulation and the more commonly used Nitsche formulation.

An important open issue for upcoming work is to lift \Cref{assumption.discrete-velocity}, stating the stability of the discrete interface velocity. This assumption enters the definition of the $\delta$-neighbourhood used for the ghost penalty stabilisation and is therefore central in the analysis. Further, we consider the following extensions to be of interest for future research.
To extend the presented analysis of the method to a more general setting, 
the equations governing the solid's motion should include rotational motion.
This requires additional work since the rotational component of the interface
velocity $\bm{\omega}\times\bm{r}$ is no longer independent of space.
Furthermore, the ODE governing the rotational motion is not linear for general
shapes. Secondly, the bulk equations should be generalised to full fluid
equations. Finally, the geometry approximation error inherent in CutFEM could
be included in further analysis.

\bibliographystyle{siamplain}
\bibliography{literature}

\end{document}